\documentclass[11pt,reqno]{amsart}
\usepackage[english]{babel}
\usepackage{amsmath}
\numberwithin{equation}{section}
\usepackage{amsfonts}
\usepackage[usenames]{color}
\usepackage{amssymb}
\usepackage{hyperref}
\usepackage{mathrsfs}
\usepackage{caption} 
\usepackage{amsthm}
\usepackage{graphicx}
\usepackage[left=3cm,right=3cm, top=2.5cm,bottom=2.5cm,bindingoffset=0cm]{geometry}
\usepackage[final]{epsfig}
\usepackage{multicol,multirow, enumitem}

\newcommand{\norm}[1]{\left\|#1\right\|}
\newcommand{\abs}[1]{\left|#1\right|}
\theoremstyle{plain}
\newtheorem{theorem}{Theorem}[section]

\newtheorem{lemma}[theorem]{Lemma}
\newtheorem{proposition}[theorem]{Proposition}

\theoremstyle{definition}
\newtheorem{remark}[theorem]{Remark}

\title[Boundary value problems for two dimensional steady incompressible fluids]{Boundary value problems for two dimensional steady incompressible fluids}

\author[D.~Alonso-Or\'an]{Diego Alonso-Or\'an}
\address{Institut für Angewandte Mathematik, Universitat Bonn, Endenicher Allee 60, 53115 Bonn, Germany}
\email{alonso@iam.uni-bonn.de}

\author[J.~J.L. Vel\'azquez]{Juan J.L. Vel\'azquez }
\address{Institut für Angewandte Mathematik, Universitat Bonn, Endenicher Allee 60, 53115 Bonn, Germany}
\email{velazquez@iam.uni-bonn.de}

\begin{document}
\begin{abstract}
In this paper we study the solvability of different boundary value problems for the two dimensional steady incompressible Euler equation and for the magneto-hydrostatic equation. Two main methods are currently available to study those problems, namely the Grad-Shafranov method and the vorticity transport method. We describe for which boundary value problems these methods can be applied. The obtained solutions have non-vanishing vorticity. 
\end{abstract}
\maketitle

\section{Introduction}
In this paper we consider several boundary value problems for the two dimensional incompressible steady Euler equation  describing the motion of an inviscid fluid given by 
\begin{equation}\label{Euler2D:eq}
\left\lbrace
\begin{array}{lll}
v\cdot \nabla v&=-\nabla p, \quad \mbox{in } \Omega \\
\nabla\cdot v&=0,  \quad \mbox{in }  \Omega
\end{array} \right.
\end{equation}
where $u:\Omega\to \mathbb{R}^2$ is the velocity fluid vector field and $p:\Omega\to \mathbb{R}$ is the scalar pressure on a suitable domain $\Omega$. System \eqref{Euler2D:eq} can be rewritten as 
\begin{equation}\label{Euler2D:eq:omega}
\left\lbrace
\begin{array}{lll}
v \times \omega&=\nabla H, \quad \mbox{in } \Omega\\
\nabla\times v&=\omega, \quad \mbox{in } \Omega\\
\nabla\cdot v&=0,  \quad \mbox{in }  \Omega
\end{array} \right.
\end{equation}
by using the well-known identity $(v\cdot\nabla) v- \frac{1}{2}\nabla(\abs{v}^2)=-v\times \omega$. Here the function $H=p+\frac{1}{2}\abs{v}^2$ is known as the Bernoulli function and the rotational of the velocity field, $\nabla\times v=\omega$ is called vorticity field. 

Solutions to \eqref{Euler2D:eq} with $\omega=0$ are termed irrotational solutions. It is well-known that boundary problems for \eqref{Euler2D:eq} (even in the three dimensional case) reduce to boundary value problems for the Laplace equation. Indeed, if $\omega=0$ we have that $v=\nabla\psi$ and the second equation in \eqref{Euler2D:eq}  implies that $\Delta \psi=0$. Technical difficulties can arise for particular types of boundary conditions. Nevertheless the well developed theory of harmonic functions can be applied to study those problems. To construct flows with non-vanishing vorticity is more challenging and the corresponding boundary value problems have been less studied in the mathematical literature.

A system of equations which is mathematically equivalent to \eqref{Euler2D:eq} is the  set of equations describing magnetohydrostatics (MHS). This is just the system for the magnetohydrodynamic equations for incompressible fluids with zero fluid velocity, namely
\begin{equation}\label{MHS2D:current}
\left\lbrace
\begin{array}{lll}
 B\times j&=-\nabla p, \quad \mbox{in }  \Omega\\
\nabla\times B&=j,  \quad \mbox{in }  \Omega \\
\nabla\cdot B&=0,  \quad \mbox{in }  \Omega
\end{array} \right.
\end{equation}
where $B$ denotes the magnetic field, $j$ the current density and $p$ the fluid pressure.
A quick  inspection of equations \eqref{MHS2D:current} reveal the equivalence of the magnetohydrostatic equations and the Euler equations  \eqref{Euler2D:eq:omega} using the transformation of variables
\begin{align}\label{transformation:variable}
B&\leftrightarrow v, \quad  j \leftrightarrow \omega, \quad  p \leftrightarrow -H.
\end{align}
Magnetohydrostatics is relevant in a wide variety of problems in astrophysical plasmas describing  coronal field structures and stellar winds. The system \eqref{MHS2D:current} is also a central model  to the study of plasma confinement fusion, (cf. \cite{Goedbloed-Poedts-2010, Goedbloed-Poedts-2010-2, Priest-2014}). 

 It is not a priori clear for which type of boundary conditions the problem \eqref{Euler2D:eq} or  \eqref{MHS2D:current} can be solved. This issue has been considered in the seminal paper of Grad and Rubin \cite{Grad-Rubin-1958} where the authors  describe several meaningful boundary value problems related to the MHS equations in two dimensional and three dimensional cases.  The main goal of this article is to study the solvability for different types of boundary value problems for the two dimensional steady incompressible Euler equations  \eqref{Euler2D:eq} (or equivalently the MHS equations (\ref{MHS2D:current})). A relevant feature of the solutions constructed in this paper is that the vorticity $\omega$ (or the current $j$) is different from zero for generic choices of the boundary values. Since our main goal is to examine the types of boundary conditions yielding well-posedness for   \eqref{Euler2D:eq} or \eqref{MHS2D:current} we will restrict ourselves to a very particular geometric setting, namely we will assume that 
 \begin{equation}\label{domain}
 \Omega=\mathbb{S}^{1}\times (0,L),
  \end{equation}
with $L>0.$ There are several reasons to choose this particular domain. First, due to the directionality of the velocity field it is natural to impose different boundary conditions on different parts of the boundary $\partial\Omega$. More precisely, one can impose different boundary conditions on the subsets of $\partial\Omega$ where $v\cdot n >0$ or $v\cdot n<0$. However, on the points of the boundary where $v\cdot n =0$ singularities for the solutions can arise. This introduces additional technical difficulties. The analysis of these singular behaviors is interesting but they will be not considered in this paper. 
 
Notice that if $\Omega$ is as in \eqref{domain}, we have $\partial\Omega=\partial\Omega_{+}\cup\partial\Omega_{-}$ where $\partial\Omega_{+}=\mathbb{S}^{1}\times\{L\}$ and $\partial\Omega_{-}=\mathbb{S}^{1}\times\{0\}$. In all the cases considered in this manuscript we will impose different types of boundary conditions on $\partial\Omega_{+}, \partial\Omega_{-}$. Since $\partial\Omega_{+}\cap \partial\Omega_{-}=\emptyset$ it is possible to impose 
boundary conditions which guarantee  $v\cdot n\neq 0$ at all the points in $\partial \Omega$. This is not case if we consider domains $\Omega$ with a connected boundary $\partial\Omega$ due to the fact that $\mbox{div }v=0$ on $\Omega$. 

The results of this paper can be easily generalized for domains $\Omega= \{(x_1,x_2) : \gamma_1(x_1)< x_2< \gamma_2(x_1) \}$
where $\gamma_j$ are smooth functions satisfying the periodicity condition $\gamma_{j}(x_1+1)=\gamma_{j}(x_1)$ for $j=1,2$. In this case we will look for solutions $(v,p)$ such that $v(x_1+1,x_2)=v(x_1,x_2)$ and $p(x_1+1,x_2)=p(x_1,x_2)$. 

The different types of boundary conditions  that we would consider in this paper are collected in the following table:
\begin{table}[h!]
  \begin{center}
    \scalebox{0.9}{
    \begin{tabular}{l|l|l} 
      \textit{BVC} & \centering{\textit{2D Euler equation \eqref{Euler2D:eq}}} & \textit{2D MHS equation \eqref{MHS2D:current}}\\
      \hline
      \hline
  \textit{(A)} & $ v\cdot n=f \ \mbox{on}  \ \partial\Omega, \quad  p + \frac{|v|^2}{2}=h \ \mbox{on} \ \partial\Omega_{-}$ & $B\cdot n=f \  \mbox{on} \ \partial\Omega, \quad p =h \ \mbox{on} \ \partial\Omega_{-}$ \\

         \textit{(B)}   & $ v\cdot n=f \  \mbox{on} \ \partial\Omega, \quad  p =h \ \mbox{on} \ \partial\Omega_{-}$& $B\cdot n=f \  \mbox{on} \ \partial\Omega, \quad p + \frac{|B|^2}{2}=h \ \mbox{at} \ \partial\Omega_{-}$ \\
        \textit{(C)}   & $p=h \ \  \mbox{on} \ \partial\Omega, \quad v\cdot n=f \ \mbox{on} \ \partial\Omega_{-} $ & $p+ \frac{|B|^2}{2}=h \  \mbox{on}  \ \partial\Omega, \quad  B\cdot n=f \ \mbox{on}  \ \partial\Omega_{-} $ \\
        \textit{(D)}  &  $p+\frac{|v|^2}{2}=h \ \  \mbox{on} \ \partial\Omega, \quad v\cdot n=f \ \mbox{on} \ \partial\Omega_{-} $  & $p=h \  \mbox{at} \ \Omega, \quad B\cdot n=f \ \mbox{at}  \ \partial\Omega_{-} $  \\
  \textit{(E)}  & $v\cdot n=f \ \mbox{on} \ \partial\Omega, \quad v\cdot t= h  \ \mbox{on}  \ \partial\Omega_{-}$ & $B\cdot n=f \ \mbox{on} \ \partial\Omega, \quad B\cdot t= h \  \mbox{on}  \ \partial\Omega_{-}$ \\
    \textit{(F)}  & 
    $ \begin{cases}
    p=h^{+} \ \mbox{on} \ \partial\Omega_{+} \\
    p+\frac{\abs{v^2}}{2}=h^{-} \ \mbox{on} \ \partial\Omega_{-}
    \end{cases} 
    $,  $v\cdot n=f^{+} \ \mbox{on} \ \partial\Omega_{+}$  &   $ \begin{cases}
    p+\frac{\abs{B}^2}{2}=h^{+} \ \mbox{on} \ \partial\Omega_{+} \\
    p=h^{-} \ \mbox{on} \ \partial\Omega^{-}	
    \end{cases} 
    $,  $B\cdot n=f^{+} \ \mbox{on} \ \partial\Omega_{+}$ \\
        \textit{(G)}  & 
    $ \begin{cases}
    p=h^{+} \ \mbox{on} \ \partial\Omega_{+} \\
    p+\frac{\abs{v^2}}{2}=h^{-} \ \mbox{on} \ \partial\Omega_{-}
    \end{cases} 
    $, $v\cdot n=f^{-} \ \mbox{on} \ \partial\Omega_{-}$  &   $ \begin{cases}
    p+\frac{\abs{B}^2}{2}=h^{+} \ \mbox{on} \ \partial\Omega_{+} \\
    p=h^{-} \ \mbox{on} \ \partial\Omega_{-}
    \end{cases} 
    $, $B\cdot n=f^{-} \ \mbox{on} \ \partial\Omega_{-}$\\
    \end{tabular}  \vspace{0.3cm}
    }
    \caption{\textit{Different types of boundary value conditions}}
      \label{tab:table1}
  \end{center}
\end{table}

We remark that each of the boundary value problems appearing in the same row in Table \ref{tab:table1} yield the same PDE problem in spite of the fact that the boundary conditions imposed for Euler equation \eqref{Euler2D:eq} and MHS  \eqref{MHS2D:current} are different. This can be seen using the variable transformation in \eqref{transformation:variable}. 

Several before mentioned boundary value conditions have a simple interpretation from the physical point of view, since we prescribed either the inflow and outflow fluxes and the pressures or the Bernoulli function on parts or the full boundary. 

We notice that in  \cite{Grad-Rubin-1958} the study of the boundary value problems \textit{(A)} and \textit{(E)} in Table \ref{tab:table1} has been posed for the MHS equations \eqref{MHS2D:current} as well as additional boundary value problems in three dimensions. Moreover, the authors also suggest an iteration scheme to solve these boundary value problems but so far the precise conditions for convergence of the iterative method has not been studied in detail. Nevertheless, the method has been seen to be successful for constructing Beltrami fields in 3D which are particular pressureless solutions of the Euler equation \eqref{Euler2D:eq:omega} (or equivalently magnetic pressureless solutions for the MHS \eqref{MHS2D:current}),  see \cite{ABM-1999,Bineau-1972,EPS-2018}.

\subsection*{Previous results}
In order to solve boundary value problems for the steady Euler or the MHS equations two main methods have been considered in the literature: the Grad-Shafranov method \cite{Grad-1967,Safranov-1966} and the \textit{vorticity transport method} introduced by Alber \cite{Alber-1992}.

The method of Grad-Shafranov is in principle restricted to two dimensional settings or to problems which can be reduced to two dimensions using symmetries (for instance axisymmetric or toroidal symmetries). We  briefly describe the main idea behind in the particular situation of the two dimensional steady Euler equation although the method can be adapted to MHS. Due to the incompressibility condition, there exists a stream function $\psi$ such that the velocity field $v=\nabla^{\perp}\psi=(-\frac{\partial\psi}{\partial y},\frac{\partial\psi}{\partial x})$ and therefore equation \eqref{Euler2D:eq:omega} is given by
\begin{equation}\label{equation:GS:1}
\Delta \psi\cdot \nabla \psi=\nabla H,
\end{equation}
where $H=p+\frac{1}{2}\abs{v}^2$  is the Bernoulli function.  Hence, \eqref{equation:GS:1} implies that there exists a function $F$ such that $H=F(\psi)$ and then \eqref{equation:GS:1} yields also 
\begin{equation}\label{equation:GS:2}
\Delta \psi= F'(\psi) \mbox{ in } \Omega.
\end{equation}
Therefore the analysis of the steady Euler equation has been reduced to the study of the elliptic equation \eqref{equation:GS:2} for which a huge number of techniques are available. The essential difficulty regarding \eqref{equation:GS:2} is to determine the function $F$ from the boundary conditions. It turns out that this is possible for some of the boundary conditions collected in Table \ref{tab:table1}. Indeed, using that $v=\nabla^{\perp}\psi$ we have that $v\cdot n=\pm\frac{\partial \psi}{\partial s}$ where $s$ is the arc-length associated to the boundary. The sign depends on the orientation chosen for the curve. Hence, suppose by definiteness that  $H$ and $v\cdot n$ are known in the same subset of the boundary, say in $\partial\Omega_{-}$. Then we can determine $\psi$ in $\partial\Omega_{-}$ (up to an additive constant) and since $H=F(\psi)$ we can also obtain the function $F$. If the additional boundary condition imposed in $\partial\Omega_{+}$ gives enough information on $\psi$ to have a well-defined problem for \eqref{equation:GS:2}, we can then determine the function $\psi$ in $\Omega$ solving \eqref{equation:GS:2} with the boundary conditions obtained  for $\psi$ in $\partial\Omega_{-},\partial\Omega_{+}$. This will be the case for the problems \textit{(A)},\textit{(D)},\textit{(G)} in Table \ref{tab:table1}. 

Clearly, when applying this procedure some technicalities will arise in order have well-defined functions $F$ and uni-valued functions $\psi$. These issues will be considered in detail in Section \ref{S:2}. As stated above the main shortcoming of the Grad-Shafranov method is that its application is restricted to two dimensional settings. Nevertheless, an extension of the method to construct rotational solutions to the three dimensional steady Euler equation in an unbounded domain $(0,L)\times \mathbb{R}^2$ with periodic flows in the unbounded directions was recently treated in \cite{Buffoni-Wahlen-2019}. Employing the so called Clebsch variables the velocity is written as $v=\nabla f \times \nabla g$ to derive an elliptic non-linear system and perform a Nash-Moser scheme to solve it. Furthermore, ideas closely related to the method of Grad-Shafranov have been recently applied to study rigidity and flexibility properties solutions of the steady Euler equation in \cite{Hamel-Nadirashvili-2017,Hamel-Nadirashvili-2019,CDG-2020}.

An alternative method to obtain solutions with non-vanishing vorticity for the steady Euler equation \eqref{Euler2D:eq} was introduced in Alber \cite{Alber-1992}. More precisely, he studied the three dimensional version of the problem \textit{(A)} in Table \ref{tab:table1} which requires an additional boundary condition. In particular, he constructed solutions where the velocity field $v$ can be splitted into $v=v_0+V$ where $v_0$ is an irrotational solution to \eqref{Euler2D:eq} and $V$ a small perturbation. 
The boundary value problem for the Euler equations is reduced to a fixed point problem for a function $V$ combining the fact that the vorticity satisfies a suitable transport equation and that the velocity can be recovered from the vorticity using the Biot-Savart law. This idea will be discussed later in more detail.

Alber's method works in particular domains $\Omega$ satisfying a geometrical constraint relating $\partial\Omega$ and $v$. A key assumption that is needed is that if $x_{0}\in\partial\Omega$ and $v(x_{0})\cdot n(x_0)=0$, then the stream line of $v$ crossing through $x_0$ is completely contained on the boundary $\partial\Omega$.

 The boundary conditions prescribed in \cite{Alber-1992} for the three dimensional case are the normal component of the velocity field on the boundary, i.e. ($v\cdot n$ on $\partial\Omega$), as well as the normal component of the vorticity $\mbox{curl }v\cdot n$ and the Bernoulli function $p+\frac{\abs{v}^2}{2}$  on the inflow set $\partial \Omega_{-}=\partial\Omega\cap \{ v\cdot n<0\}$. A straightforward computation shows that the two boundary conditions imposed on the inflow set $\partial \Omega_{-}$, prescribe completely the vorticity in in  $\partial \Omega_{-}$. It is possible to determine the vorticity in any point of the domain $\Omega$ using the fact that it satisfies a first order differential equation by means of the  characteristics method.

Since Alber's result, there have been several generalizations and extensions. In \cite{Tang-Xin-2009}, the authors provide a modification of Alber's technique to construct solutions to the three dimensional steady Euler equation where the base flow does not have to satisfy Euler equation and the boundary conditions are given by $\mbox{curl }v=av+b$ on the inflow set $\partial\Omega$ for certain values of $a$ and $b$ satisfying compatibility conditions. Extensions of Alber's results to compressible flows with non-smooth domains have been obtained in \cite{Molinet-1999}. Solutions to the three dimensional steady Euler equation with boundaries meeting at right angles have been constructed in \cite{Seth-2020}. An illustrative example of this situation are curved pipes domains.

\subsection*{Main results: novelties and key ideas}
We describe here the main results and key ideas to construct solutions with non-vanishing vorticity for the two dimensional incompressible steady Euler equation \eqref{Euler2D:eq} with the different boundary value conditions collected on Table \ref{tab:table1}.

The Grad-Shafranov method allows to solve the boundary value problems \textit{(A)},\textit{(D)} and \textit{(G)}. 
The case \textit{(A)} has been already considered in \cite{Arnold-Khesin-1999, Seth-2016}. In this paper we will discuss the application of the Grad-Shafranov approach in Section \ref{S:2}. 

Hereafter, we will adapt the arguments of  Alber \cite{Alber-1992} to solve the boundary value problems \textit{(B)}, \textit{(C)} and  \textit{(G)}, in Table \ref{tab:table1}. As indicated above, the proof builds on a ground flow $v_{0}$ solving \eqref{Euler2D:eq} which is perturbed by function $V$  which will be determined by solving a fixed point of for a suitable operator $\Gamma$.  The idea to construct the operator $\Gamma$, relies on two building blocks: a transport type problem and a div-curl problem. 
The former consists in finding a unique function $\omega$ for a given $v$ and $\omega_{0}$ satisfying 
\begin{equation}\label{tp:omega}(\textit{TP})
\left\lbrace
\begin{array}{lll}
v\cdot\nabla \omega =0 \ \mbox{at} \ \Omega, \\
\omega = \omega_{0} \ \mbox{at} \ \partial \Omega_{-}.
\end{array} \right.
\end{equation} 
The value of $\omega_0$ is chosen in a particular way in order to get a solution which satisfies the boundary value conditions we want to deal with.

The second building block relies on finding  a unique $W$ which solves 
 \begin{equation}\label{div:curl:problem}(\textit{DCP})
\left\lbrace
\begin{array}{lll}
\nabla\times W= \omega, \ \mbox{at} \ \Omega \\
\mbox{div } W=0, \ \mbox{at} \ \Omega \\ 
W\cdot n = g, \   \mbox{ at } \partial \Omega 
\end{array} \right.
\end{equation} 
for a given $\omega$.
We will restrict ourselves to the very particular domains in \eqref{domain}. As indicated before, the main reason for that is that in general open domains $\Omega\subset \mathbb{R}^2$ with smooth connected boundary $\partial\Omega$ there are necessarily boundary points $x_{0}\in \partial\Omega$ such that $v(x_{0})\cdot n(x_0)=0$, which will be termed from now on as tangency points. Integrating by characteristics (assuming that the vector field $v$ is oriented in such a way that such problem is solvable), the solutions of \eqref{tp:omega} develop singularities in the derivatives that makes difficult to solve the combined problem \eqref{tp:omega}-\eqref{div:curl:problem} by means of a fixed point argument. In order to avoid this difficulty Alber restrict himself to smooth domains $\Omega$ with Lipschitz boundaries $\partial \Omega$ satisfying the following condition: if a vector field $v$ has a tangency point at $x_0$, then the whole stream line of $v$ crossing $x_0$ is contained on the boundary $\partial\Omega$ (see equations (1.22)-(1.23) in \cite{Alber-1992}).

A benefit of working with our domains \eqref{domain} is that they do not have tangency points and hence this difficulties can be ignored. As a drawback, since the domains \eqref{domain} are not simply connected, we need to impose topological constraints to our vector fields in order to have well defined problems. In particular, problem \eqref{div:curl:problem} cannot be reduced in general to a Laplace equation unless the flux of $W$ along a vertical line is zero. However, this can be achieved by adding a horizontal constant vector.

An important observation and difference with the work of Alber, is that we use Hölder spaces instead of Sobolev spaces to construct our solutions.  In the case treated by Alber, the vorticity $\omega_{0}$ at the boundary can be readily obtained from the boundary value given in the problem. However this is not the case for the problems \textit{(B)} and \textit{(C)}. In those cases $\omega_0$ is part of the solution that is obtained by means of the fixed point argument. More precisely, the value of $\omega_0$  is given in terms of $W$  and its derivatives at the boundary $\partial\Omega_{-}$, where $W$ solves \eqref{div:curl:problem}. If the estimates for $W$ are given in terms of the Sobolev spaces, we obtain less regularity 
for $\omega_0$ due to the classical Trace Theorem. The vorticity $\omega$ is now computed on the whole domain $\Omega$ using the transport equation \eqref{tp:omega} and this does not  give any gain of regularity for $\omega$. As a consequence when we recover a new function $W$ using again \eqref{div:curl:problem} which has less regularity that we had initially. Due to this it is not possible to close a fixed point argument. To solve this obstruction we make use of Hölder spaces where this difficulty vanishes. Furthermore, in the case \textit{(G)}, it is possible to obtain the vorticity $\omega_0$ in terms of the boundary values. However the normal veloctiy in $\partial\Omega_{+}$ it is not known and it must be obtained using a fixed point argument as in the cases describe above.

It is worth to notice that seemingly the case \textit{(D)} cannot be solved using Alber's method due to the fact that a loss of regularity takes place when one tries to reformulate the problem as a fixed point. In this case the lost of regularity is an essential difficulty that cannot fixed even with the use of Hölder spaces. See Section \ref {S:3:4} for a detailed explanation of this fact.

The case \textit{(E)} which is one of the boundary value problems suggested in the pioneering work of Grad and Rubin \cite{Grad-Rubin-1958}, does not seem amenable to any of the two methods indicated above and will be studied in a forthcoming work by means of completely different techniques.

Similarly, the case \textit{(F)}  seems to pose essential difficulties for both methods. Indeed, we cannot applied the Grad-Shafranov method since we do not know the value of $H$ and $v\cdot n$ at the same part of the boundary. On the other hand, we cannot apply Alber's method due to the loss of regularity similar as it happens in the case \textit{(D)}.
 
To the best of our knowledge several of the boundary value problems in the Table \ref{tab:table1} have not been studied in the scientific literature.  One of the main goals of this paper is to clarify which sets of boundary conditions yield well-posed problems.

\subsection*{Notation} We will use the following notation throughout the manuscript. We recall that we are working on a domain $\Omega= \mathbb{S}^1\times (0,L)$ with $L>0$.  Let $C_{b}(\Omega)$ be the set of bounded continuous functions on $\Omega.$ For any bounded continuous function and $0<\alpha<1$ we call $f$ uniformly Hölder continuous with exponent $\alpha$ in $\Omega$ if the quantity
$$ \left[ f \right]_{\alpha,\Omega}:= \displaystyle\sup_{x\neq y; x,y\in \overline{\Omega}} \frac{\abs{f(x)-f(y)}}{\abs{x-y}^\alpha}$$
is finite. However, this is just a semi-norm and hence in order to work with Banach spaces we define the space of Hölder continuous functions as 
$$ C^{\alpha}(\Omega)=\{ f\in C_{b}(\Omega): \norm{f}_{C^\alpha(\Omega)}< \infty\},$$
equipped with the norm
$$ \norm{f}_{C^{\alpha}(\Omega)}:=\displaystyle \sup_{x\in \overline{\Omega}}\abs{f(x)}+  \left[ f \right]_{\alpha,\Omega}.$$
Similarly, for any non-negative integer $k$ we define the Hölder spaces $C^{k,\alpha}(\Omega)$ as 
$$ C^{\alpha}(\Omega)=\{ f\in C^{k}_{b}(\Omega): \norm{f}_{C^{k,\alpha}(\Omega)}< \infty\},$$
equipped with the norm
$$ \norm{f}_{C^{k,\alpha}(\Omega)}=\displaystyle\max_{\abs{\beta}\leq k }\sup_{x\in \overline{\Omega}}\abs{\partial^{\beta}f(x)}+\displaystyle\sum_{\beta=k}\left[\partial^{\beta}f \right]_{\alpha,\Omega}.$$
Notice that in the definitions above the Hölder regularity holds up to the boundary, i.e in $\overline{\Omega}$.
We omit in the functional spaces whether we are working with scalars or vectors fields, this is $C^{k,\alpha}(\Omega,\mathbb{R})$ or $C^{k,\alpha}(\Omega,\mathbb{R}^2)$ and instead just write $C^{k,\alpha}(\Omega)$. Moreover, we will identify $\mathbb{S}^1$ with the interval $[0,1]$ and the functions $f\in C^{k,\alpha}(\mathbb{S}^1)$ , $k=0,1,2...$, $\alpha\in [0,1)$ with the functions $f\in C^{k,\alpha}([0,1])$  satisfying that $f^{\ell}(0)=f^{\ell}(1)$ for $ \ell=0,1,\dots,k$. Notice that this space of functions can also be identified with the space $f\in C^{k,\alpha}(\mathbb{R})$ such that $f(x+1)=f(x)$.

\subsection*{Plan of the paper}
In Section \ref{S:2} we show how to solve the boundary value cases \textit{(A)},\textit{(D)} and \textit{(G)} using the Grad-Shafranov approach. Next, in Section \ref{S:3} we introduce the vorticity transport method
and apply it to construct solutions to the steady Euler equations for the boundary value cases \textit{(B)},\textit{(C)} and \textit{(G)}. In the last section, Section \ref{S:4}, we translate the statements of the results  shown for the Euler equation \eqref{Euler2D:eq:omega} in the case of the MHS equations \eqref{MHS2D:current}. 

\section{The Grad-Shafranov approach}\label{S:2}
In this section we will use the Grad-Shafranov method to construct solutions with non-vanishing vorticity to the Euler equation for the boundary value problem \textit{(D)}. Although the method is also valid to tackle the case \textit{(G)}, we will give the details in that case using the fixed point method.  As explained in the introduction, the Grad-Shafranov approach reduces the existence problem for the Euler equation \eqref{Euler2D:eq} to the study of a simpler elliptic equation $ \Delta  \psi= F'(\psi)$ where $v=\nabla^{\perp}\psi$ and $F(\psi)$ is an unknown function related to the Bernoulli function $H$ that we need to determine using the boundary value conditions. 
\subsection{Boundary value problem  \textit{(D)}  for the steady Euler equation}\label{S:2.1}
\begin{theorem}\label{theorem:GS:1}
Let $f\in C^{1,\alpha}(\partial\Omega_{-}), h^{-}\in C^{1,\alpha}(\partial\Omega_{-})$, $h^{+}\in C^{1,\alpha}(\partial\Omega_{+})$ and $h^{+}=h^{-}\circ T$ where $T:\mathbb{S}^1\to \mathbb{S}^1$ is a given diffeomorphism with $C^{2,\alpha}$ regularity.  Then if $f>0$ for $x\in \partial \Omega_{-}$, there exists a solution $(v,p)\in C^{{1,\alpha}}(\Omega)\times C^{{1,\alpha}}(\Omega)$ solving the Euler equation \eqref{Euler2D:eq} such that 
\begin{equation}\label{boundary:value:type:GS}
v\cdot n =f \mbox{on} \ \partial \Omega_{-}, \ p+\frac{\abs{v}^2}{2}=h^{-}  \ \mbox{on} \ \partial \Omega_{-} \mbox{ and } p+\frac{\abs{v}^2}{2}=h^{+}  \ \mbox{on} \ \partial \Omega_{+}.
\end{equation}
Moreover, there exists $\delta>0$ such that if 
\begin{equation}\label{smallness:uniqueness:GS:1}
\norm{h^{-}}_{C^{{1,\alpha}}(\partial \Omega_{-})}+\norm{f}_{C^{{1,\alpha}}(\partial \Omega)}\leq  \delta ,
\end{equation}
the solution $(v,p)$ is unique. 
\end{theorem}

\begin{proof}
First we let $f(x)$ and $h^{-}(x)$ be extended periodically to the whole real line $\mathbb{R}$. Then we define $\psi_{-}(x)=\int_{0}^{x}f(s) \ ds$ and notice that the function $\psi(x)$ is invertible since $f>0$ in $\partial\Omega_{-}$, this is there exists a function $\xi$ such that $X(\xi)=\psi^{-1}_{-}(\xi)$ for all $\xi\in\mathbb{R}$. Moreover,
\begin{equation}\label{period:psi:inverse}
\psi_{-}(x+1)=\psi_{-}(x)+J, \ \mbox{ and } \psi_{-}^{-1}(\xi+J)=\psi_{-}^{-1}(\xi)+1,
\end{equation}
where $J=\int_{0}^{1} f(s) ds>0$. Next, we define the function 
 $F(\xi)=h^{-}(\psi^{-1}(\xi))$ for every $\xi\in\mathbb{R}$ which is a periodic function of period $J$. Indeed,
 $$F(\xi+J)=h^{-1}(\psi^{-1}_{-}(\xi+J))=h^{-1}(\psi^{-1}_{-}(\xi)+1)=h^{-1}(\psi^{-1}_{-}(\xi))=F(\xi),$$
 where we have used the periodicity of $h^{-1}$ and $\psi^{-1}_{-}$ in \eqref{period:psi:inverse}. Notice that the function $F\in C^{1,\alpha}$ since $f,h^{-}\in C^{1,\alpha}(\partial\Omega_{-})$. Finally the function $\psi_{+}(x):\mathbb{S}^1\to \mathbb{S}^1$ given by $\psi_{+}(x)=(\psi_{-}\circ T)(x)$. With these definitions and constructions at hand, we are interested in solving the following elliptic boundary value problem 
\begin{equation}\label{GS:problem1}
\left\lbrace
\begin{array}{lll}
\Delta \psi= F'(\psi), \ \mbox{in }  \Omega \\
\psi(x,0)=\psi_{-}(x), \mbox{on } \partial \Omega_{-}\\
\psi(x,L)=\psi_{+}(x), \ \mbox{on } \partial \Omega_{+} \\
\psi(0,y)=\psi(1,y)+J, \ \mbox{for } y\in(0,L).
\end{array} \right.
\end{equation} 
In order to obtain a minimization problem in the whole manifold $\mathbb{S}^{1}\times [0,L]$ we make the following change of variables  $\phi=\psi-Jx$ where the new function $\phi$ solves
\begin{equation}\label{GS:problem2}
\left\lbrace
\begin{array}{lll}
\Delta \phi= F'(Jx+\phi), \ \mbox{in } \Omega \\
\phi(x,0)=\psi_{-}(x) -Jx, \ \mbox{on } \partial \Omega_{-} \\
\phi(x,L)=\psi_{+}(x) -Jx, \ \mbox{on } \partial \Omega_{+} \\
\phi(1,y)=\phi(0,y), \ \mbox{for } y\in(0,L)
\end{array} \right.
\end{equation} 
The new functions $F'(Jx+\phi), \psi_{-}(x) -Jx, \psi_{+}(x)-Jx$  defined on the manifold $\mathbb{S}^{1}\times [0,L]$ are periodic. To show the existence of solutions to \eqref{GS:problem2} we use the classical variational calculus theory. To that purpose, we introduce the following energy functional
\begin{equation}\label{functional:I}
I[\phi]= \frac{1}{2}\int_{\Omega} \abs{\nabla \phi}^{2}+ \int_{\Omega} F(Jx+\phi) \ dx,  
\end{equation}
the admissible space of functions
\begin{equation}
\mathcal{A}=\{ \phi\in H^{1}(\Omega): \phi(x,0)=\psi_{-}(x) -Jx \ \mbox{on } \partial \Omega_{-}, 
\phi(x,L)=\psi_{+}(x) -Jx \ \mbox{on } \partial \Omega_{+} \ \mbox{in the trace sense}\},
\end{equation}
and set 
\begin{equation}\label{minimizer}
\bar{\phi}=\mbox{arg min}\{ I[\phi]: \phi\in \mathcal{A}\}. 
\end{equation}
It is well-known that this minimizing problem has at least one solution and moreover that it is a weak solution to \eqref{GS:problem2} which verify also the boundary conditions in the trace sense, see \cite{Evans2010}. Moreover, since $F\in C^{1,\alpha}$ and $\psi_{-},\psi_{+}\in C^{2,\alpha}$, an application of standard elliptic regularity theory in Hölder spaces shows that $\phi\in C^{2,\alpha}(S^1\times [0,L])$, (cf. \cite{Gilbarg-Trudinger-2001}). Therefore, by construction we have that 
$$v=\nabla^{\perp} \psi \in C^{1,\alpha}(\Omega) \mbox{ and }p=F(\psi)-\frac{\abs{\nabla^{\perp}\psi}^{2}}{2}\in C^{1,\alpha}(\Omega),$$ 
solves the Euler equation \eqref{Euler2D:eq} with boundary value conditions \eqref{boundary:value:type:GS} concluding the proof. To show uniqueness, let $\phi^1$ and $\phi^2$ be two different solutions to \eqref{GS:problem2} and set $\widehat{\phi}=\phi^1-\phi^2$. Then we have that 
\begin{equation}\label{GS:problem3}
\left\lbrace
\begin{array}{lll}
\Delta \widehat{\phi}= F'(Jx+\phi^1)- F'(Jx+\phi^2), \ \mbox{in }  \Omega\\
\widehat{\phi}(x,0)=0,\ \mbox{on } \partial \Omega_{-} \\
\widehat{\phi}(x,L)=0, \ \mbox{on } \partial \Omega_{+}  \\
\widehat{\phi}(1,y)=\widehat{\phi}(0,y), \ \mbox{for } y\in(0,L).
\end{array} \right.
\end{equation} 
Applying classical regularity theory for elliptic problems (cf. \cite{Gilbarg-Trudinger-2001}) we have that
\begin{equation}\label{estimate:GS:problem1}
\norm{\widehat{\phi}}_{C^{2,\alpha}(\Omega)}\leq C \norm{F'(Jx+\phi^1)- F'(Jx+\phi^2)}_{C^{\alpha}(\Omega)}.
\end{equation}
Using the smallness condition on $h^{-}, f$ in \eqref{smallness:uniqueness:GS:1} and the fact that $F= h^{-}\circ\psi^{-1}$ we infer that
\begin{equation}\label{estimate:GS:problem2}
\norm{F'(Jx+\phi^1)- F'(Jx+\phi^2)}_{C^{\alpha}(\Omega)} \leq C \delta \norm{\widehat{\phi}}_{C^{2,\alpha}(\Omega)},
\end{equation}
and hence combining estimates \eqref{estimate:GS:problem1}-\eqref{estimate:GS:problem2} yields $\widehat{\phi}=0$, i.e. $ \phi^1=\phi^2$.
\end{proof}
\begin{remark}
We should notice that we only need that the normal component of $v$ is different than zero at the inflow boundary $\partial\Omega_{-}$ and hence the vector field $v$ could have generally null points. On the other hand, we should remark that the diffeomorphism $T$ does not have to be unique for certain elections of the functions $h^{-}$ and $h^{+}$ , and therefore for each $T$ we would have a different solution $(v,p)$. Indeed, we can have two different diffeomorphism $T$ such that $T(A)=\hat{A}, T(B)=\hat{B}$ or $T(A)=\hat{B}, T(B)=\hat{A}$ as in Figure \ref{Figura1}.
\begin{figure}[h]
\includegraphics[width=10cm]{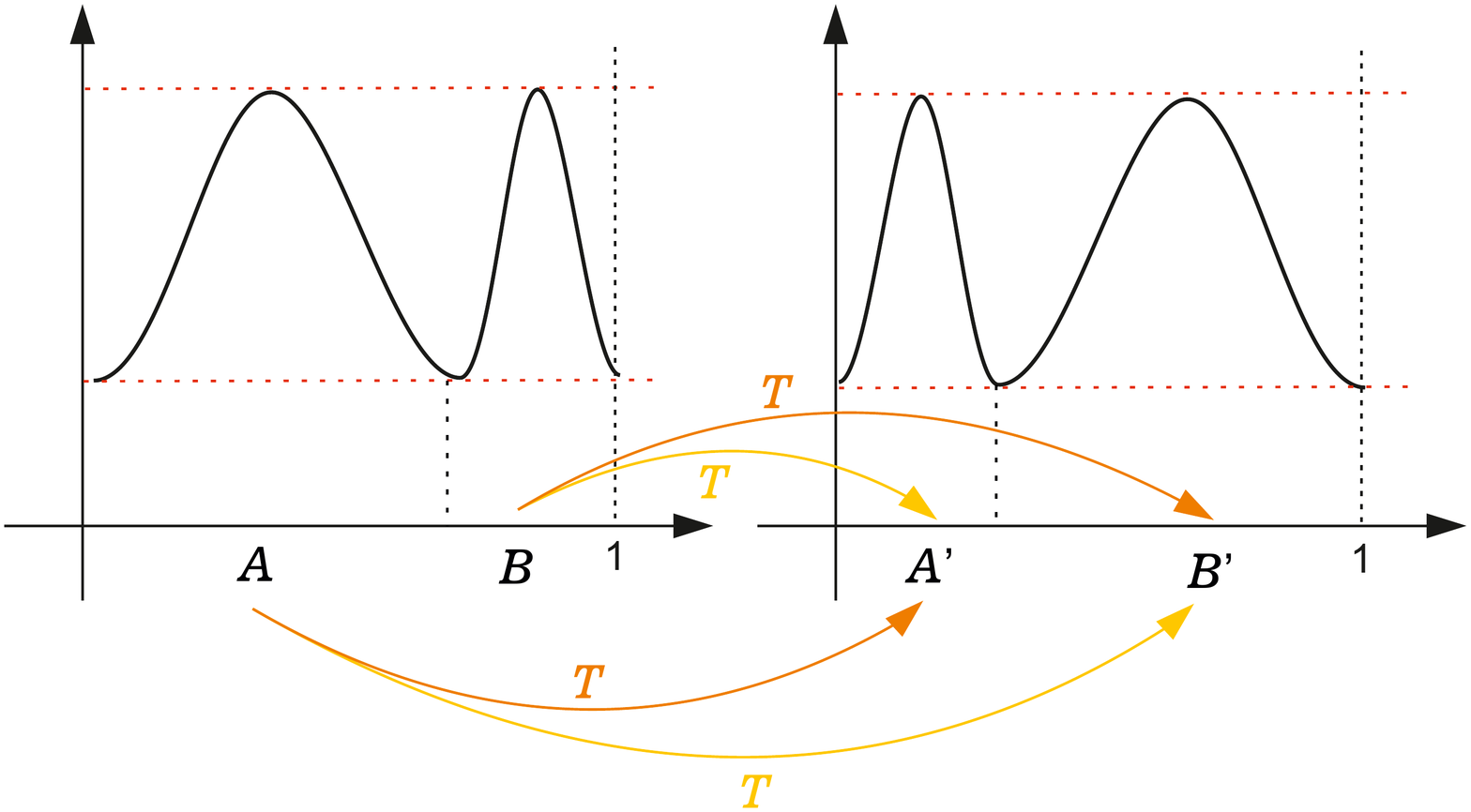}
\caption{Different diffeomorphism $T$}
\label{Figura1}
\end{figure}

\end{remark}

\begin{remark}[Boundary value problem \textit{(G)} using Grad-Shafranov] We give here a brief explanation on how to modify the arguments above in order to treat the boundary value problem \textit{(G)}. However since this case will be tackled later with the vorticity transport method (cf. Subsection \ref{S:3:3}) we do not provide the full details but the idea behind. The statement of the result reads:
\begin{theorem}\label{theorem:GS:2}
Let $f\in C^{1,\alpha}(\partial\Omega_{-}), h^{-}\in C^{1,\alpha}(\partial\Omega_{-})$, $h^{+}\in C^{1,\alpha}(\partial\Omega_{+})$. Then there exists a constant $\delta>0$ such that if 
\begin{equation}\label{smallnesss:GS:2}
\norm{h^{-}}_{C^{{1,\alpha}}(\partial \Omega_{-})}+\norm{h^{+}}_{C^{{1,\alpha}}(\partial \Omega_{+})}+\norm{f}_{C^{{1,\alpha}}(\partial \Omega_{-})}\leq  \delta,
\end{equation}
a unique solution $(v,p)\in C^{{1,\alpha}}(\Omega)\times C^{{1,\alpha}}(\Omega)$ solving the Euler equation \eqref{Euler2D:eq} such that 
\begin{equation}
v\cdot n =1+f \ \mbox{on} \ \partial \Omega_{-}, \ p+\frac{\abs{v}^2}{2}=h^{-}  \ \mbox{on} \ \partial \Omega_{-} \mbox{ and } p=h^{+}  \ \mbox{on} \ \partial \Omega_{+}.
\end{equation}
\end{theorem}
Mimicking the same argument as in the proof of Theorem \ref{theorem:GS:1}, we have reduced the existence of solutions to the following boundary value problem 
\begin{equation}\label{GS:problem:3:G}
\left\lbrace
\begin{array}{lll}
\Delta \psi= F'(\psi), \ \mbox{in }  \Omega \\
\psi(x,0)=\psi_{-}(x), \ \mbox{on } \partial \Omega_{-} \\
h^{+}+\frac{\abs{\nabla^{\perp}\psi}^{2}}{2}= F(\psi), \ \mbox{on } \partial \Omega_{+} \\
\psi(0,y)=\psi(1,y)+J, \ \mbox{for } y\in(0,L)
\end{array} \right.
\end{equation} 
where $\psi_{-}(x)=\int_{0}^{x}f(s) \ ds$, $F(\xi)=h^{-}(\psi^{-1}(\xi))$ for every $\xi\in\mathbb{R}$. Due to the nonlinear character of the boundary condition on $\partial\Omega_{+}$ it is not a priori clear if the problem \eqref{GS:problem:3:G} can be solved for arbitrary functions $\psi_{-},F$ and $ h_{+}$. However, under the smallness assumption \eqref{smallnesss:GS:2}, we can linearize the equation $h^{+}+\frac{\abs{\nabla^{\perp}\psi}^{2}}{2}= F(\psi) \ \mbox{ on } \partial \Omega_{+}$ for a suitable perturbation $\psi=\psi_{0}+\delta \phi$ where $\psi_{0}=x$ and $\delta\ll 1$ and solve the resulting problem by means of a fixed point argument. Actually analogous perturbative arguments will be applied recurrently in the rest of the paper. 
\end{remark}

\begin{remark}[Boundary value problem \textit{(A)} using Grad-Shafranov]
In the case of the boundary value problem \textit{(A)}, the construction of solutions to the Euler equation \eqref{Euler2D:eq} reduces also to the study of an elliptic equation which can be treated with classical calculus of variations tools. This case has been studied before in the literature in \cite{Arnold-Khesin-1999, Seth-2016} and we would not provide more details here. 

\begin{theorem}\label{theorem:GS:4}
Let $f^{-}\in C^{1,\alpha}(\partial\Omega_{-}), f^{+}\in C^{1,\alpha}(\partial\Omega_{+})$ and $h^{-}\in C^{1,\alpha}(\partial\Omega_{-})$. Then if $f^{-}>0$ for $x\in \partial \Omega_{-}$, there exists a solution $(v,p)\in C^{{1,\alpha}}(\Omega)\times C^{{1,\alpha}}(\Omega)$ solving the Euler equation \eqref{Euler2D:eq} such that 
\begin{equation}
v\cdot n =f^{-} \ \mbox{on } \ \partial \Omega_{-}, v\cdot n =f^{+} \ \mbox{on} \ \partial \Omega_{+} \mbox{ and }  p+\frac{\abs{v}^2}{2}=h^{-}  \ \mbox{on} \ \partial \Omega_{-}.
\end{equation}
Moreover, there exists $\delta>0$ such that if 
\begin{equation}
\norm{f^-}_{C^{{1,\alpha}}(\partial \Omega_{-})}+\norm{f^+}_{C^{{1,\alpha}}(\partial \Omega_{+})} +\norm{h^{-}}_{C^{{1,\alpha}}(\partial \Omega_{-})}\leq  \delta ,
\end{equation}
the solution $(v,p)$ is unique. 
\end{theorem}
\end{remark}

\section{The vorticity transport method}\label{S:3}
In this section we will apply the fixed point method approach to construct non-vanishing solutions to the Euler equation for the boundary value problems \textit{(B)},\textit{(C)} and \textit{(G)}. In order to avoid repetition we will show cases \textit{(B)} and \textit{(C)} in full detail and just provide a sketch of the proof   for case \textit{(G)}  highlighting the main differences.

\subsection{Boundary value problem  \textit{(B)}  for the steady Euler equation }\label{S:31}
We will construct solutions to the Euler equation with boundary conditions \textit{(B)} using a suitable modification of the vorticity transport method introduced by Alber \cite{Alber-1992}. Let us state the result precisely:
\begin{theorem}\label{Th:case2}
 Let $\Omega=\{ (x,y)\in \mathbb{S}^{1}\times (0,L)\}$, with $L>0$ and $\alpha\in (0,1)$.  Suppose that $(v_0,p_0)\in C^{{2,\alpha}}(\Omega)\times C^{{2,\alpha}}(\Omega)$ is a solution of \eqref{Euler2D:eq} with $\bar{v}^{2}_{0}=\displaystyle\inf_{(x,y)\in\bar{\Omega}} \abs{v^2_{0}(x,y)}> 0$ and $\mbox{curl }v_0=0$. For $\mathcal{C}= \{(0,y), y\in[0,L] \}$ we have that the integral $\int_{\mathcal{C}} (v_{0}\cdot n) \ dS$ is a real  constant that we will denote as $J_0$. There exist $\epsilon>0$ , $M>0$ sufficiently small as well as $K>0$ such that for $v_0$ as above with $ \norm{v^1_0}_{C^{2,\alpha}(\Omega)}\leq\epsilon $ and 
$h\in C^{{2,\alpha}}(\partial \Omega_{-})$, $f\in C^{{2,\alpha}}(\partial \Omega)$ and $J\in \mathbb{R}$  satisfying 
\begin{equation}\label{smallnes:bound}
\norm{h}_{C^{{2,\alpha}}(\partial \Omega_{-})}+\norm{f}_{C^{{2,\alpha}}(\partial \Omega)}+ \abs{J-J_0} \leq KM ,
\end{equation} 
and 
\begin{equation}\label{f:condition}
\int_{\partial\Omega_{-}}f \ dS=  \int_{\partial\Omega_{+}}f \ dS,
\end{equation}
there exists a unique $(v,p)\in C^{{2,\alpha}}(\Omega)\times C^{{2,\alpha}}(\Omega)$ to \eqref{Euler2D:eq}  with $\norm{v-v_0}_{C^{2,\alpha}(\Omega)}\leq M$ such that
\begin{equation}\label{boundary:value:type1}
v\cdot n =v_{0}\cdot n+f \ \mbox{on} \ \partial \Omega, \ p=p_0+h  \ \mbox{on} \ \partial \Omega_{-} \mbox{ and } \int_{\mathcal{C}}v\cdot n \ dS= J.
\end{equation}
The constants $M,K,$ as well as $\epsilon,$ depend  only on $\alpha,L, \bar{v}^{2}_{0}$.
\end{theorem}

\begin{remark}
Notice that we have chosen our base flow $v_0$ to be irrotational. From the mathematical point of view the strategy of the proof is sufficiently flexible to cover the case when $\mbox{curl }v_0$ is different than zero but sufficiently small. However, for each specific boundary condition it is not obvious if suitable rotational solutions exist.
\end{remark}

\begin{remark}
It is not a priori clear whether the smallness assumption on $v_0^{1}$ in Theorem \ref{Th:case2}  can be removed. This is due to the fact that a crucial step of the argument is to solve the equation \eqref{omega:o:type1} in order to obtain the value of the vorticity at $y=0$. The term $-\frac{1}{v_0^2}\partial_{x}(v_{0}^{1}V^{1})$ on the right hand side in \eqref{omega:o:type1} is linear in $V$ and therefore we cannot treat it perturbatively if $v^1_0$ is not small. If $v=v_{0}+V$ with $V \ll 1$ it is natural to try a linearization approach that yields a problem of the form 
\begin{equation}\label{ellip:hyper:problem}
\left\lbrace
\begin{array}{lll}
\Delta \psi= \omega(x,y), \ \mbox{in }  \Omega \\
v_0\cdot \nabla \omega=Q_{1}, \ \mbox{in }  \Omega \\
\omega(x,0)=\partial_{x}(v_{0}^{1}V^{1})+Q_2, \mbox{ on } \partial \Omega_{-} \\
\end{array} \right.
\end{equation} 
where $Q_1, Q_2$ contain terms that are quadratic in $V$ or small source terms due to the boundary data and $V=\nabla^{\perp}\psi$. The existence of an operator yielding $(V,\omega)$ in terms of $Q_1,Q_2$ can fail if the homogeneous problem obtained setting $Q_1=Q_2=0$ has non-trivial solutions. 
\end{remark}

\begin{remark}
The curve $\mathcal{C}= \{(0,y), y\in[0,L] \}$ along we fixed the flux $J$ can be chosen in a more general way. Indeed, we can choose two different curves $C_1$ and $C_2$ in which we have the same flux $J$ if we impose that $\int_{\Gamma_1}v\cdot n = \int_{\Gamma_2}v\cdot n$ as in Figure \ref{Figura2}.
\begin{figure}[h]
\includegraphics[width=8cm]{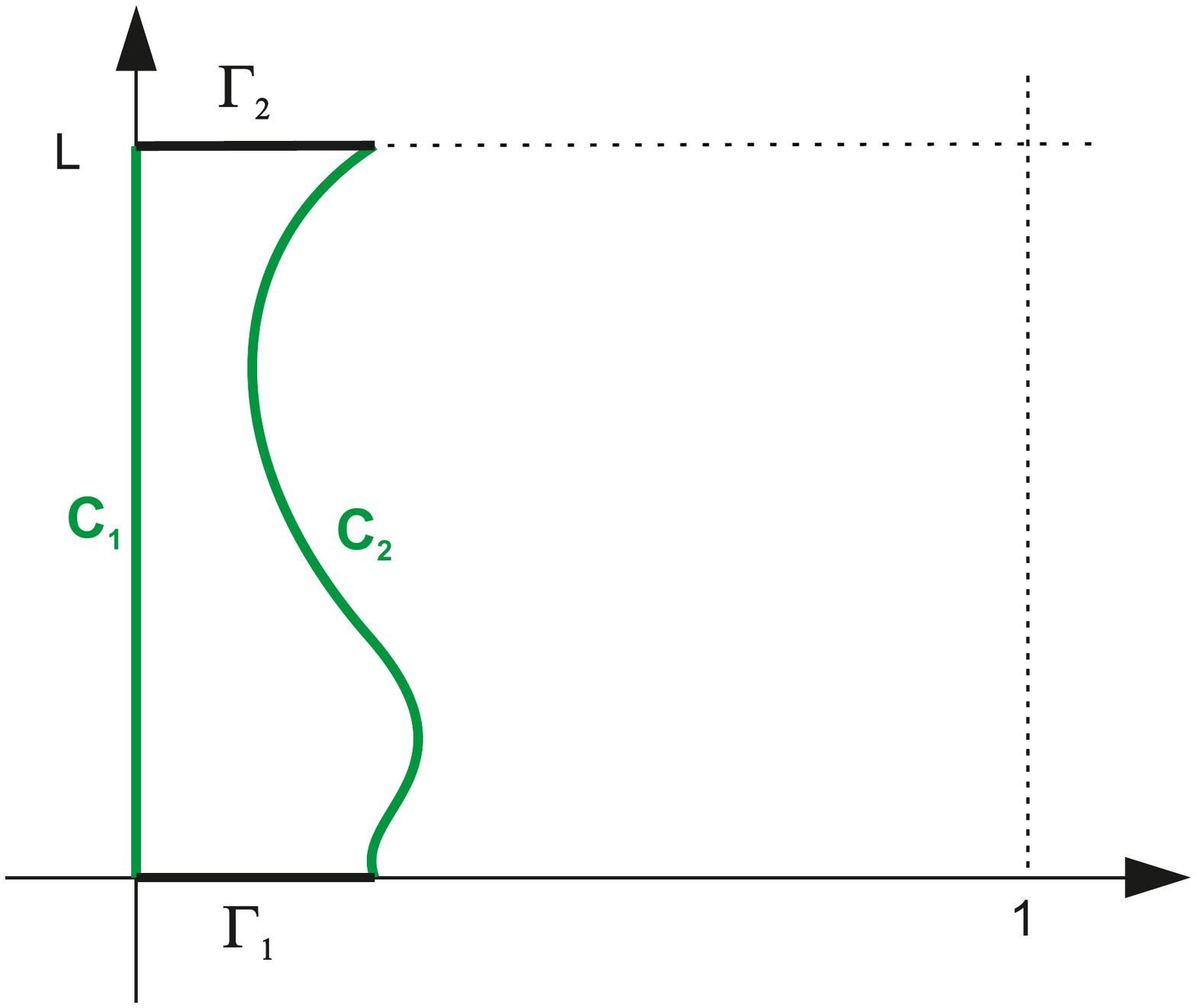}
\caption{General curves $\mathcal{C}$.}
\label{Figura2}
\end{figure}

\end{remark}

As we have mentioned in the introduction, the proof is based on defining an adequate operator $\Gamma$ on a subspace of $C^{2,\alpha}(\Omega)$ which has a fixed point $V$ such that $v=v_{0}+V$ is a solution to  \eqref{Euler2D:eq} and \eqref{boundary:value:type1}. To that purpose, let us define the following subspace $C_{\star}^{2,\alpha}(\Omega)$ given by 
\begin{equation}\label{C:2:alpha}
C_{\star}^{2,\alpha}(\Omega)=\{ g\in C^{2,\alpha}(\Omega): \mbox{div } g=0 \ \mbox{in } \Omega\}. 
\end{equation}
For any, $M>0$, let us denote by $B_{M}$ the closed ball of in $C_{\star}^{2,\alpha}(\Omega)$ with radius $M$, i.e.,  
\begin{equation}\label{B:M2:alpha}
B_{M}=\{ g\in C_{\star}^{2,\alpha}(\Omega): \norm{g}_{C^{2,\alpha}(\Omega)}\leq M \}.
\end{equation}

\begin{remark}\label{remark:lower:bound:v}
Notice that  if $v_0$ satisfies the hypothesis of Theorem \ref{Th:case2}, there exists a fixed constant $0<M_0<\frac{\bar{v}^{2}_{0}}{2}$ such that $v=v_{0}+V$ with $V\in B_{M_{0}}(\Omega)$ satisfies 
\begin{equation}\label{lower:bound:v}
\displaystyle\inf_{(x,y)\in \overline{\Omega}} \abs{v^2(x)}\geq  \bar{v}^{2}_{0}-M_{0}>\frac{\bar{v}^{2}_{0}}{2},
\end{equation} 
since $V^2\leq  \norm{V}_{C^{2,\alpha}(\Omega)} \leq M_{0}$. Throughout the article we will always assume (sometimes will be written explicitly) that  $M\leq M_{0}$, so that the lower bound \eqref{lower:bound:v} holds true for any function $ V\in B_{M}(\Omega)$.
\end{remark}

\subsubsection{The building blocks: the transport problem and the div-curl system}\label{S:2:1}
In this subsection we will provide regularity results and show several estimates regarding the hyperbolic transport problem and the div-curl problem which are the building blocks to construct the operator $\Gamma$. Notice that the results will be used not only to solve  the boundary value problem (\textit{B}) but will be instrumental to construct solutions to  boundary value problems \textit{(C)}, \textit{(G)} treated in this article (cf.  Section \ref{S:32} and Section \ref{S:3:3}). 

Before proceeding any further, let us show a regularity result for the trajectories associated to a vector field. We define the flow of a continuous and bounded vector field $b(x,y)$ as the map $X:\Omega\to \mathbb{S}^{1}$ which satisfies 
\begin{equation}\label{ode:flow:def}
\left\lbrace
\begin{array}{lll}
\frac{\partial X }{\partial y}(a,y)&= b(X(a,y),y) \\
X(a,0)&=a.
\end{array}\right.
\end{equation}
We refer to $a$ as the particle label, since it marks the beginning point of the path $a\to X(a,y)$. 
\begin{lemma}\label{flow:map:bounds} 
Let assume that $b\in C^{1,\alpha}(\Omega)$. Then, there exists a unique solution $X\in C^{1,\alpha}(\Omega) $  solving \eqref{ode:flow:def}.  Moreover, the following estimates are satisfied:
\begin{align}
\norm{X}_{C^{1,\alpha}(\Omega)}\ &\leq C\left(L, \norm{b}_{C^{1,\alpha}(\Omega)}\right), \label{bound:c1} \\
\norm{X^{-1}}_{C^{1,\alpha}(\Omega)} & \leq  C\left(L, \norm{b}_{C^{1,\alpha}(\Omega)}\right).\label{bound:c2}
\end{align}
\end{lemma}
\begin{proof}
To derive \eqref{bound:c1} we need to estimate $\frac{\partial X}{\partial y},\frac{\partial X}{\partial a}$ as well as their Hölder norms. Notice that a bound for the $\norm{\frac{\partial X}{\partial y}}_{L^{\infty}(\Omega)}$ follows directly by \eqref{ode:flow:def}.  Furthermore, standard results of differentiability with respect to parameters of ordinary differential equations (cf. \cite{Coddington-Levison-1955}) yield
\begin{align}
\frac{\partial X}{\partial a}(a,y) =\mbox{exp}\left(\int_{0}^{y}\frac{\partial b}{\partial X}(X(a,s),s)\ ds \right), \label{bound:c1:a}
\end{align}
hence an estimate for  $\norm{\frac{\partial X}{\partial a}}_{L^{\infty}(\Omega)}$ follows directly. To estimate the Hölder norm of $\frac{\partial X}{\partial a}$ we compute the difference 
\begin{align}
\abs{\dfrac{ \partial X}{ \partial a}(a_1,y)-\dfrac{\partial X}{ \partial a}(a_2,y)} &\leq   C \int_0^{y} \abs{\frac{\partial b}{\partial X}(X(a_1,s),s)-\frac{\partial b}{\partial X}(X(a_2,s),s)} ds  	\nonumber \\
&\leq  C  \int_{0}^{y} \norm{b}_{C^{1,\alpha}}\abs{X(a_1,y)-X(a_2,y)}^{\alpha} \nonumber \\
&\leq  CL  \norm{b}_{C^{1,\alpha}}\norm{\frac{\partial X}{\partial a}}_{L^\infty}^{\alpha}|a_{1}-a_{2}|^{\alpha}.
\end{align}
Therefore, 
\begin{equation}\label{holder:norm:a} 
 \frac{\abs{\frac{ \partial X }{ \partial a}(a_1,y)-\frac{ \partial X }{ \partial a}(a_2,y)}}{\abs{a_1-a_2}^\alpha}  \leq CL  \norm{b}_{C^{1,\alpha}}\norm{\frac{\partial X}{\partial a}}_{L^\infty}^{\alpha} \leq C\left(L, \norm{b}_{C^{1,\alpha}(\Omega)}\right), 
 \end{equation}
for $a_{1},a_{2}\in \mathbb{S}^1, y\in (0,L)$. On the other hand, given that $\norm{\frac{\partial X}{\partial y}}_{L^\infty}$  and $\norm{\frac{\partial X}{\partial a}}_{L^\infty}$ are bounded we have that the function $b(X(a,y),y)$ is Lipschitz in both variables $a,y$ ( and then Hölder on both variables $a,y$). Hence,
 \begin{equation}\label{holder:norm:both} 
 \frac{\abs{\frac{\partial X }{\partial y}(a_{1},y_1)-\frac{\partial X }{\partial y}(a_{2},y_2)}}{\abs{a_1-a_2}^\alpha+|y_{1}-y_{2}|^\alpha}\leq C\left(L, \norm{b}_{C^{1,\alpha}(\Omega)}\right), \quad a_{1},a_{2}\in \mathbb{S}^1, y\in (0,L).
 \end{equation} 
Combining bounds  \eqref{holder:norm:a}-\eqref{holder:norm:both} we infer that $\norm{X}_{C^{1,\alpha}(\Omega)}  \leq C\left(L, \norm{b}_{C^{1,\alpha}(\Omega)}\right)$ and hence \eqref{bound:c1} is satisfied. 
 
%
%
%

We are left to show \eqref{bound:c2}. To that purpose, notice that estimate \eqref{bound:c1:a} implies that $\frac{\partial X}{\partial a}>\mbox{exp}(L\norm{b}_{C^1(\Omega)})^{-1}>0$ and therefore the mapping $a\to X(a,y)$ is invertible for any $y\in (0,L)$ with inverse $X^{-1}(x,y)$ such that $X(X^{-1}(x,y),y)=x$. Moreover, 
\begin{equation}\label{inverse:exp}
 \frac{\partial X^{-1}}{\partial x}(x,y)=\frac{1}{\frac{\partial X}{\partial a}(X^{-1}(x,y),y)}, \quad   \frac{\partial X^{-1}}{\partial y}(x,y)=-\frac{\frac{\partial X}{\partial y}(X^{-1}(x,y),y)}{\frac{\partial X}{\partial a}(X^{-1}(x,y),y)}.
 \end{equation}
Using that $\norm{\frac{\partial X}{\partial x}}_{L^{\infty}(\Omega)},\norm{\frac{\partial X}{\partial y}}_{L^{\infty}(\Omega)}$ are bounded, it then follows that
 $X^{-1}$ is Lipschitz for $(x,y)\in \Omega$.
 
Moreover we can also estimate the Hölder semi-norms of $\frac{\partial X^{-1}}{\partial x},\frac{\partial X^{-1}}{\partial y}$ as above using that $\norm{X}_{C^{1,\alpha}(\Omega)}  \leq C(L, \norm{b}_{C^{1,\alpha}(\Omega)})  $, then the following bound holds
$$ \norm{X^{-1}}_{C^{1,\alpha}(\Omega)} \leq C(L, \norm{b}_{C^{1,\alpha}(\Omega)}), $$
proving \eqref{bound:c2} and concluding the proof.
\end{proof}

Next,  we will derive Hölder estimates for solutions to the hyperbolic transport type problem given by
\begin{equation}\label{transport:problem}(\textit{TP})
\left\lbrace
\begin{array}{lll}
(v_{0}+V)\cdot\nabla \omega =0 \ \mbox{in} \ \Omega, \\
\omega = \omega_{0} \ \mbox{on} \ \partial \Omega_{-}
\end{array} \right.
\end{equation} 
which is the first building block to construct the fixed point operator. 
\begin{proposition}\label{TP:prop}
Let $v_0$ be as in Theorem \ref{Th:case2} and $0<M_0<\frac{\bar{v}^{2}_{0}}{2}$. Then for every $M\leq M_{0}$, $V\in B_{M}(\Omega)$ and $\omega_0\in  C^{1,\alpha}(\partial\Omega_{-})$,  there exists a unique $\omega\in C^{1,\alpha}(\Omega)$ solving \eqref{transport:problem}. Moreover, there exists a constant $C=C(\alpha,\Omega,L,\bar{v}^{2}_{0})>0$ such that the following estimate holds
\begin{equation}\label{estimate:TP}
\norm{\omega}_{C^{1,\alpha}(\Omega)} \leq C  \norm{\omega_0}_{C^{1,\alpha}(\partial\Omega_{-})}. 
\end{equation}
Furthermore, let $\omega^{1},\omega^{2}\in C^{\alpha}(\Omega)$ be  two different solutions to  \eqref{transport:problem} with  $V$ given by $ V^{1},V^{2}$  respectively. Then
\begin{equation}\label{estimate:difference:TP}
\norm{\omega^{1}-\omega^{2}}_{C^{\alpha}(\Omega)} \leq  C \left(\norm{\omega^1_{0}-\omega^2_{0}}_{C^{\alpha}(\partial\Omega_{-})}+ \norm{\omega^{1}_{0}}_{C^{1,\alpha}(\partial\Omega_{-})}\norm{V^1-V^2}_{C^{\alpha}(\Omega)}\right)
\end{equation}
where $C=C(\alpha,L,\bar{v}^{2}_{0})>0$.
\end{proposition}
\begin{remark}
It is important to emphasize that the positive constants $C$ depend only on the following fixed quantities $C=C(\alpha, L,\bar{v}^{2}_{0})$. Note also that  it might change from line to line.  For exposition's clearness we will avoid writing explicitly the constants dependences along the proofs throughout the manuscript.
\end{remark}
\begin{proof}
We can  solve equations  \eqref{transport:problem} using the integral curves of $v=v_0+V$. More precisely, the explicit solution to \eqref{transport:problem} is given by 
$$ \omega(x,y)=\omega_{0}(X^{-1}(x,y)),$$
where $X^{-1}$ is the inverse of the mapping $a\to X(a,y)$ solving the ordinary differential equation \eqref{ode:flow:def} with $b(x,y)=\frac{(v^1_{0}+V^{1})(x,y)}{(v^2_{0}+V^2)(x,y)}$. Since $0<M_0<\frac{\bar{v}^{2}_{0}}{2}$ and $M\leq M_{0}$,  we have that $\displaystyle\inf_{(x,y)\in\overline{\Omega}} \abs{v^2_{0}+V^2}>0$ (cf.  Remark \ref{remark:lower:bound:v}), hence $b(x,y)$ has $C^{1,\alpha}(\Omega)$ regularity and satisfies the bound
\begin{equation}\label{bound:b:alpha}
 \norm{b(x,y)}_{C^{1,\alpha}(\Omega)}\leq C.
\end{equation}
Using Lemma \ref{flow:map:bounds}, we have that there exists a unique  $X\in C^{1,\alpha}(\Omega)$ solving the system \eqref{ode:flow:def} with inverse $X^{-1}$. Therefore, invoking the estimate \eqref{bound:c2} in Lemma \ref{flow:map:bounds} and the bound \eqref{bound:b:alpha} we have that
\begin{align}\label{bound:vort:0}
\norm{ \omega(x,y)}_{C^{1,\alpha}(\Omega)}= \norm{\omega_{0}(X^{-1}(x,y))}_{C^{1,\alpha}(\partial\Omega)} &\leq C  \norm{\omega_{0}}_{C^{1,\alpha}(\partial\Omega_{-})}\left( 1+\norm{\nabla X^{-1}}_{C^{\alpha}(\Omega)}\right)^{1+\alpha} \nonumber \\
& \leq C \norm{\omega_{0}}_{C^{1,\alpha}(\partial\Omega_{-})}.
\end{align}
To show \eqref{estimate:difference:TP}, we use the notation $\widehat{\omega}=\omega_{1}-\omega_{2}=$ and $\widehat{V}=V_{1}-V_{2}$.  From \eqref{transport:problem}, we have that 
\begin{equation}\label{difference:transport:2}
\left\lbrace
\begin{array}{lll}
(v_{0}+V_{2})\cdot\nabla \widehat{\omega} = -\widehat{V}\cdot\nabla\omega_{1} \quad  \mbox{in} \ \Omega, \\
\widehat{\omega}=\widehat{\omega_{0}}\quad  \mbox{on } \partial \Omega_{-}.
\end{array}\right.
\end{equation}
Solving \eqref{difference:transport:2} using characteristics we have that
\begin{equation}\label{sol:diff:transport}
 \widehat{\omega}(x,y)=\widehat{\omega_{0}}(X^{-1}(x,y)) - \int_0^y \left(\frac{\widehat{V}\cdot\nabla\omega_{1}}{v^2_{0}+V^2_2}\right)(X(X^{-1}(x,y),s),s) \  ds
\end{equation}
where $X$ solves the ordinary differential equation \eqref{ode:flow:def} with $b(x,y)=\frac{(v^1_{0}+V^{1}_{2})(x,y)}{(v^2_{0}+V^2_{2})(x,y)}$.
Therefore, we infer that
\begin{eqnarray*}
\norm{\widehat{\omega}(x,y)}_{C^{\alpha}(\Omega)} &\leq & C \norm{\widehat{\omega_{0}}}_{C^{\alpha}(\partial\Omega_{-})}\left(1+\norm{\nabla X^{-1}}_{C^{0}(\Omega)}\right) \\
&& +  C \norm{\displaystyle\int_0^y \left(\frac{\widehat{V}\cdot\nabla\omega_{1}}{v^2_{0}+V^2_2}\right)\left(X(X^{-1}(x,y),s),s\right) \  ds}_{C^{\alpha}(\Omega)} = I_1+I_2.
\end{eqnarray*}
Applying estimate \eqref{bound:c2} in Lemma \ref{flow:map:bounds} and bound \eqref{bound:b:alpha}, we can estimate the first term as before
\begin{equation}\label{I_1}
\abs{I_1}\leq C \norm{\widehat{\omega}_{0}}_{C^{\alpha}(\partial\Omega_{-})}.
\end{equation}
Next, we estimate the second term $I_2$. To that purpose, let us define $\phi(x,y)=\displaystyle\int_{0}^{y} \mathcal{H}(x,y,s) ds$ for any function $\mathcal{H}\in C^\alpha(\Omega\times (0,L))$. Then we have that
\begin{equation}\label{estimate:phi:h}
\abs{\phi(x_{1},y)-\phi(x_2,y)} \leq L \norm{\mathcal{H}}_{C^\alpha} \abs{x_{1}-x_{2}}^{\alpha}, \quad y\in (0,L),
\end{equation}
and for any $0\leq y_1\leq y_2 \leq  L$
\begin{eqnarray}\label{estimate:phi:h:2}
\abs{\phi(x,y_{1})-\phi(x,y_{2})}&\leq & \left(\int_{0}^{y_{2}}\norm{\mathcal{H}}_{C^\alpha} \ ds \right) \abs{y_{1}-y_{2}}^{\alpha} + \int_{y_{1}}^{y_2} \norm{\mathcal{H}}_{L^\infty} \ ds \nonumber \\
 &\leq &  L\norm{\mathcal{H}}_{C^\alpha}|y_{1}-y_{2}|^{\alpha}+\norm{\mathcal{H}}_{L^{\infty}}|y_{1}-y_{2}|, \quad x\in [0,1].
\end{eqnarray}
Hence, we have shown that  $\norm{\phi}_{C^{\alpha}} \leq C(L) \norm{\mathcal{H}}_{C^{\alpha}}$ for $ x\in [0,1], y\in (0,L)$. Applying this result for $\mathcal{H}=\left(\frac{\widehat{V}\cdot\nabla\omega_{1}}{v^2_{0}+V^2_2}\right)\left(X(X^{-1}(x,y),s),s\right)$, we conclude that
\begin{eqnarray}\label{I_2}
\abs{I_2} \leq  \norm{\left(\frac{\widehat{V}\cdot\nabla\omega_{1}}{v^2_{0}+V^2_2}\right)(X(X^{-1}(x,y),s),s)}_{C^\alpha(\Omega)}  &\leq&  C\norm{\frac{\widehat{V}\cdot\nabla\omega_{1}}{v^2_{0}+V^2_2}}_{C^\alpha(\Omega)} \nonumber \\
&\leq & C\norm{\omega_1}_{C^{1,\alpha}(\Omega)}\norm{\widehat{V}}_{C^{\alpha(\Omega)}}  \nonumber \\
&\leq & C\norm{\omega_{0,1}}_{C^{1,\alpha}(\partial\Omega_{-})}\norm{\widehat{V}}_{C^{\alpha}(\Omega)},
\end{eqnarray}
where in the second inequality we have used the fact that $X$ and $X^{-1}$ are Lipschitz, so bound  \eqref{bound:c1} in Lemma \ref{flow:map:bounds} applies and in the latter we have invoked bound \eqref{bound:vort:0}.

Hence, collecting estimates \eqref{I_1}-\eqref{I_2} we deduce that
\begin{equation}
\norm{\widehat{\omega}(x,y)}_{C^{\alpha}(\Omega)} \leq C \left(\norm{\widehat{\omega}_{0}}_{C^{\alpha}(\partial\Omega_{-})}+ \norm{\omega_{0,1}}_{C^{1,\alpha}(\partial \Omega_{-})}\norm{\widehat{V}}_{C^{\alpha}(\Omega)}\right),
\end{equation}
which shows \eqref{estimate:difference:TP} and concludes the proof.
\end{proof}

On the other hand, we have the following result for the div-curl problem:

\begin{proposition}\label{DCP:prop}
For every $j\in \mathbb{R}$, $\omega\in C^{1,\alpha}(\Omega)$ and $f\in C^{2,\alpha}(\partial\Omega)$ satisfying \eqref{f:condition},  there exists a unique solution $W\in C^{2,\alpha}(\Omega)$ solving
\begin{equation}\label{DCP:problem}
\left\lbrace
\begin{array}{lll}
\nabla \times W = \omega, \quad \mbox{in } \Omega \\ 
\mbox{div } W =0,  \quad \mbox{in } \Omega  \\
W \cdot n = f ,  \quad \mbox{on } \partial \Omega \\
\int_{\mathcal{C}}W\cdot n \ dS= j.
\end{array}\right.
\end{equation}
Moreover, the solution satisfies the inequality
\begin{equation}\label{estimate:DCP}
\norm{W}_{C^{2,\alpha}(\Omega)} \leq C \left(\norm{\omega}_{C^{1,\alpha}(\Omega)}+ \norm{f}_{C^{2,\alpha}(\partial\Omega)} + \abs{j}\right),
\end{equation}
where  $C=C(L,\alpha)>0$. 
\end{proposition}
\begin{proof}
To solve system \eqref{DCP:problem}, we examine the following auxiliary problem, namely
\begin{equation}\label{stream:function}
\left\lbrace
\begin{array}{lll}
\Delta \psi= \omega, \quad x\in \Omega \\ 
\psi= -j+\displaystyle\int_{0}^{x} (f(\xi)-A) \ d\xi  ,  \quad x\in \partial \Omega_{+} \\
\psi= \displaystyle\int_{0}^{x} (f(\xi)-A) \ d\xi  ,  \quad x\in \partial \Omega_{-} \\
\end{array}\right.
\end{equation}
where $A=\int_{\partial\Omega_{+}} f \ dS=\int_{\partial\Omega_{-}} f  \ dS.$ Notice that this particular choice of $A$, yields a well-defined uni-valued function  $\psi$. Moreover, if $\psi$ is a solution to \eqref{stream:function} of sufficient regularity (actually $\psi\in C^{3,\alpha}(\Omega)$), we get a solution to \eqref{DCP:problem} by defining $W=(0,A)+\nabla^{\perp}\psi$ where  $\nabla^{\perp}\psi=(-\frac{\partial \psi}{\partial y},-\frac{\partial \psi}{\partial x})$. We only verify the last condition in \eqref{DCP:problem}, since the other ones are straightforward to check. Indeed, we have that
$$\displaystyle\int_{\mathcal{C}}W\cdot n \ dS=-\int_{0}^{L} \frac{\partial \psi}{\partial y}(0,y) \ dy=-\psi(0,L)+\psi(0,0)=j.$$ Hence applying classical regularity theory for elliptic problems (cf. \cite{Gilbarg-Trudinger-2001}) we infer that there exists a unique solution $W\in C^{2,\alpha}(\Omega)$ satisfying the following bound:
$$ \norm{W}_{C^{2,\alpha}(\Omega)}\leq C \left(\norm{\omega}_{C^{1,\alpha}(\Omega)}+ \norm{f}_{C^{2,\alpha}(\partial\Omega)} + \abs{j}\right),$$
as desired.
\end{proof}

\subsubsection{The fixed point argument and construction of the solution}\label{S:2:2}
The proof is based on defining an adequate operator $\Gamma$ which has a fixed point $V$ such that $v=v_{0}+V$ is a solution to  \eqref{Euler2D:eq} and \eqref{boundary:value:type1}.

We define the operator $\Gamma: B_{M}(\Omega) \to C_{\star}^{2,\alpha}(\Omega)$ in two steps. First, given $V\in B_{M}(\Omega)$ we define $\omega\in C^{1,\alpha}(\Omega)$ solving the following the transport type problem 
\begin{equation}\label{transport:problem:type1}
\left\lbrace
\begin{array}{lll}
(v_{0}+V)\cdot\nabla \omega =0, \ \mbox{in} \ \Omega \\
\omega = \omega_{0}, \ \mbox{on} \ \partial \Omega_{-}
\end{array} \right.
\end{equation} 
with  $\omega_0$ given by
\begin{equation}\label{omega:o:type1}
\omega_{0}=\partial_{x}f+ \frac{1}{f+v^2_{0}}\left( -\partial_{x}h+\partial_{x}(v^{1}_{0}V^1)+\frac{1}{2}\partial_{x}\abs{V^1}^2+f \partial_{y}v^{1}_{0}\right),\  \forall x\in\partial\Omega_{-},
\end{equation}
where $v_{0}=(v^1_{0},v^2_{0})$ and $V=(V^1,V^2)$.  As a second step, we define $W\in C^{2,\alpha}(\Omega)$ as the unique solution to the following div-curl problem
 \begin{equation}\label{div:curl:problem:type1}
\left\lbrace
\begin{array}{lll}
\nabla\times W= \omega, \ \mbox{in} \ \Omega \\
\mbox{div } W=0, \ \mbox{in} \ \Omega \\ 
 W\cdot n=f, \ \mbox{on} \ \partial \Omega \\
\displaystyle \int_{\mathcal{C}} W\cdot n \  dS=J-J_{0}.
\end{array} \right.
\end{equation} 
Thus we define $\Gamma(V)=W$.

\begin{lemma}\label{Lemma:contraction} Let  $v_0$  be as in Theorem \ref{Th:case2}, $h\in C^{2,\alpha}(\partial\Omega_{-})$ and $ f\in C^{2,\alpha}(\partial\Omega)$ satisfying \eqref{f:condition}.   There exists $\delta_{0}=\delta_{0}(\Omega,\alpha, L, \bar{v}^{2}_{0})$ and $M_{0}\in (0, \frac{\bar{v}^{2}_{0}}{2})$ such that if $M\in (0, M_{0})$ and 
\begin{equation}
\epsilon \leq \delta_{0}, \ \norm{h}_{C^{2,\alpha}(\partial\Omega_{-})}+\norm{f}_{C^{2,\alpha}(\partial\Omega)}\leq M_{0}\delta_{0}  \mbox{ and } \abs{J-J_{0}} \leq M_{0}\delta_{0},
\end{equation}
then $\Gamma( B_{M}(\Omega))\subset  B_{M}(\Omega)$. Moreover, the operator $\Gamma$ has a unique fixed point in $B_{M}(\Omega).$ 
\end{lemma}
\begin{remark}
Notice that the operator $\Gamma$ is not a compact operator, and hence we cannot prove the result  applying Schauder's fixed point theorem. 
\end{remark}

\begin{proof}
The well-definiteness of the operator $\Gamma$ follows directly from Proposition \ref{TP:prop} and Proposition \ref{DCP:prop}. Indeed, since $v_{0}$ satisfies the hypothesis in Theorem \ref{Th:case2} and  $h\in C^{2,\alpha}(\partial\Omega_{-}), f\in C^{2,\alpha}(\partial\Omega)$, we have that  $\omega_{0}$  as in \eqref{omega:o:type1} satisfies  $\omega_{0}\in C^{1,\alpha}(\partial\Omega_{-})$. Applying Proposition \ref{TP:prop}, there exists a unique $\omega\in C^{1,\alpha}(\Omega)$ satisfying  \eqref{transport:problem:type1}. Therefore Proposition \ref{DCP:prop} gives a unique $W\in C_{\star}^{2,\alpha}(\Omega)$ satisfying \eqref{div:curl:problem:type1}.

We now show that the operator  $\Gamma$ maps $ V\in B_{M}(\Omega)$ into itself.  Using inequality \eqref{estimate:DCP} with $j=J-J_{0}$ and then \eqref{estimate:TP}, we have that
\begin{eqnarray}\label{estimate:dcp:1}
\norm{\Gamma(V)}_{C^{2,\alpha}(\Omega)} &\leq & C \left(\norm{\omega}_{C^{1,\alpha}(\Omega)}+\norm{f}_{C^{2,\alpha}(\partial \Omega)}+\abs{J-J_{0}}\right) \nonumber \\
&\leq & C \left(\norm{\omega_0}_{C^{1,\alpha}(\partial\Omega_{-})}+\norm{f}_{C^{2,\alpha}(\partial \Omega)}+\abs{J-J_{0}}\right).
\end{eqnarray}
In addition, since $V\in B_{M}(\Omega)$, the function $\omega_{0}$ defined in  \eqref{omega:o:type1} can be estimated as follows
\begin{eqnarray}\label{estimate:omega:1}
\norm{\omega_{0}}_{C^{1,\alpha}(\partial\Omega_{-})}\ &=&\norm{\partial_{x}f+ \frac{1}{f+v^2_{0}}\left( \partial_{x}h-\partial_{x}(v^{1}_{0}V^1)+\frac{1}{2}\partial_{x}\abs{V^1}^2+f \partial_{y}v^{1}_{0}\right)}_{C^{1,\alpha}(\partial\Omega_{-})}  \nonumber \\ 
&\leq & C \left( \norm{h}_{C^{2,\alpha}(\partial\Omega_{-})}+\norm{f}_{C^{2,\alpha}(\partial\Omega)} + \epsilon \norm{V}_{C^{2,\alpha}(\Omega)}+\norm{V}^2_{C^{2,\alpha}(\Omega)} + \epsilon \norm{f}_{C^{2,\alpha}(\partial \Omega)}\right) \nonumber \\ 
&\leq & C \left( \norm{h}_{C^{2,\alpha}(\partial\Omega_{-})}+\norm{f}_{C^{2,\alpha}(\partial\Omega)} +(\epsilon M_{0}+M_{0}^2)\right) \leq 
C(\delta_0M_{0}+M_{0}^2).
\end{eqnarray}
Combining estimates \eqref{estimate:dcp:1}-\eqref{estimate:omega:1} we have that
\begin{equation}\label{estimate:gamma:2}
\norm{\Gamma(V)}_{C^{2,\alpha}(\Omega)}\leq  C(\delta_{0}M_{0}+M_{0}^2).
\end{equation}
Choosing $C\delta_{0}\leq\frac{1}{4}$ and $CM_{0}\leq \frac{1}{4}$, we obtain $\Gamma( B_{M}(\Omega))\subset  B_{M}(\Omega)$, for $M\in (0,M_{0})$.

We now claim that $B_{M}(\Omega)$ endowed with the topology $C^{1,\alpha}$ is a complete metric space which we will denote as $(B_{M}(\Omega),\norm{\cdot}_{C^{1,\alpha}})$. We also claim that $\Gamma:(B_{M}(\Omega),\norm{\cdot}_{C^{1,\alpha}})\to (B_{M}(\Omega),\norm{\cdot}_{C^{1,\alpha}})$ is a contraction mapping.

 In order to prove the first claim, it is enough to show that  $B_{M}(\Omega)$ is a closed subset of $C^{1,\alpha}(\Omega)$. Assume that the sequence $\{ V_{n}\}\subset B_{M}(\Omega)$ converges to some $V$ strongly in $C^{1,\alpha}(\Omega)$. Since  $\{ V_{n}\} $ is bounded in the $C^{2,\alpha}(\Omega)$, Arzela-Ascoli theorem implies that there exists a subsequence $\{ V_{{n}_{k}}\} $ which converges in $C^{2}(\Omega)$ to some $U$. Hence, $V=U\in C^{2}(\Omega)$ and $\{ V_{n}\}\ \to V\in C^2(\Omega)$.  Moreover, it turns out that the limit $V\in B_{M}(\Omega)$. Indeed, using the definition of the norm $\norm{V}_{C^{2,\alpha}(\Omega)}$ we have that 
$$\norm{V_{n}}_{C^{2}(\Omega)}+\frac{\abs{\partial^{\beta}V_{n}(x)-\partial^{\beta}V_{n}(y)}}{\abs{x-y}^\alpha}\leq M, \quad x,y\in\Omega, x\neq y,  \mbox{ for } |\beta|=2.$$
Using that $\partial^{\beta}V_n\to \partial^{\beta}V$ uniformly in $\Omega$ we can take the limit in the previous inequality to obtain 
$$\norm{V}_{C^{2}(\Omega)}+\frac{\abs{\partial^{\beta}V(x)-\partial^{\beta}V(y)}}{\abs{x-y}^\alpha}\leq M, \quad x,y\in\Omega, x\neq y,  \mbox{ for } |\beta|=2. $$
Taking the supremum in $x,y\in \Omega, x\neq y$ and the maximum in $\abs{\beta}$ we obtain that 
$\norm{V}_{C^{2,\alpha}}\leq M$, thus $V\in B_{M}(\Omega)$. 

We are just left to show our second claim, namely that $\Gamma:(B_{M}(\Omega),\norm{\cdot}_{C^{1,\alpha}})\to (B_{M}(\Omega),\norm{\cdot}_{C^{1,\alpha}})$ is a contraction.  For $V^{1},V^{2}\in  B_{M}(\Omega)$, we need to estimate the difference $\norm{\Gamma(V^1)-\Gamma(V^2)}_{C^{1,\alpha}(\Omega)}$. Due to the linearity  of the div-curl problem \eqref{div:curl:problem:type1} we can use inequality \eqref{estimate:DCP} and bound \eqref{estimate:difference:TP}  to get
\begin{align*}
\norm{\Gamma(V^1)-\Gamma(V^2)}_{C^{1,\alpha}(\Omega)} &\leq C \norm{\omega^1-\omega^2}_{C^{\alpha}(\Omega)} \\
&\leq C \left(\norm{\omega^1_{0}-\omega^2_{0}}_{C^{\alpha}(\partial\Omega_{-})}+ \norm{\omega_{0,1}}_{C^{1,\alpha}(\partial\Omega_{-})}\norm{V^1-V^2}_{C^{\alpha}(\Omega)}\right).
\end{align*}
Computing the difference $\omega^1_{0}-\omega^2_{0}$ using equation \eqref{omega:o:type1} we have that
\begin{eqnarray}\label{estimate:difference:0} 
\norm{\omega^1_{0}-\omega^2_{0}}_{C^{\alpha}(\partial\Omega_{-})} &=& \norm{\frac{1}{f+v_{0}^{2}}\left( -\partial_{x}(v^{1}_{0}(x)(V^{1}_{1}-V^{2}_{1}))-V^{1}_{2}\partial_{x}(V^{1}_{1}-V^{2}_{1})-(V^{1}_{1}-V^{2}_{1})\partial_{x}V^{1}_{1}\right)}_{C^\alpha(\partial\Omega_{-})} \nonumber \\
&\leq&  C \left(\delta_{0}M_{0}\norm{V^1-V^2}_{C^{1,\alpha}(\Omega)}+\delta_{0}M_{0}\norm{V^1-V^2}_{C^{1,\alpha}(\Omega)}\right).
\end{eqnarray}
On the other hand, we notice that applying bound \eqref{estimate:omega:1} yields
\begin{equation}\label{omega:0:1}
\norm{\omega_{0,1}}_{C^{1,\alpha}(\partial\Omega_{-})}\leq C(\delta_{0}M_{0}+M_{0}^2),
\end{equation}
and hence combining estimates \eqref{estimate:difference:0} and \eqref{omega:0:1} we  have shown that 
\begin{align*}
\norm{\Gamma(V^1)-\Gamma(V^2)}_{C^{1,\alpha}(\Omega)}\leq C(\delta_{0}M_{0}+M_{0}^2)\norm{V^1-V^2}_{C^{1,\alpha}}<\theta\norm{V^1-V^2}_{C^{1,\alpha}},
\end{align*}
where $\theta$ is strictly less than one for $C\delta_{0}\leq \frac{1}{4}$ and $CM_{0}\leq \frac{1}{4}$. Therefore,  $\Gamma:(B_{M}(\Omega),\norm{\cdot}_{C^{1,\alpha}})\to (B_{M}(\Omega),\norm{\cdot}_{C^{1,\alpha}})$ is a contraction mapping,  for $M\in(0,M_0)$.

Invoking Banach's fixed point theorem we can conclude that $\Gamma$ admits a unique fixed point 	$V\in B_{M}(\Omega)$, i.e. $\Gamma(V)=V$, which concludes the proof.
\end{proof}

\subsubsection{Proof of Theorem \ref{Th:case2}}\label{S:2:3}
Let  $\delta_{0}=\delta_{0}(\Omega,\alpha,L,\bar{v}^{2}_{0})$ and $M_{0}$ be the constants defined in Lemma \ref{Lemma:contraction}. Take $0<\epsilon\leq\delta_{0}$,  $K=\delta_{0}$ and $M\in (0,M_{0})$. Therefore,  Lemma \ref{Lemma:contraction} implies that  $\Gamma$ has a unique fixed point $V\in B_{M}(\Omega)$.

We claim  that $V\in B_{M}(\Omega)$ is a fixed point of the operator $\Gamma$, if an only if  $v=v_0+V$ is the velocity field which is a solution $(v,p)\in C^{{2,\alpha}}(\Omega)\times C^{{2,\alpha}}(\Omega)$ to \eqref{Euler2D:eq} and \eqref{boundary:value:type1}. Indeed, on the one hand assume that $V\in B_{M}(\Omega)$ is a fixed point of $\Gamma$ and write $v=v_0+V$. Then using the definition of the space $B_{M}(\Omega)$ \eqref{C:2:alpha}-\eqref{B:M2:alpha} the following properties hold
\[\mbox{div } v=\mbox{div } v_{0}+\mbox{div } V=0, \quad \mbox{in } \Omega, \quad v\cdot n =  v_{0}\cdot n+ V\cdot n= v_{0}\cdot n + f, \quad \mbox{on } \partial\Omega. \]
Moreover, from the last equation in \eqref{div:curl:problem:type1} we have that 
$$ \int_{\mathcal{C}}v\cdot n \ dS= \int_{\mathcal{C}}v_{0}\cdot n \ dS+J-J_{0}=J.$$
Since $V$ is fixed point of $\Gamma$  and $\nabla\times v_{0}=0$ we infer that
\begin{equation*}
\nabla\times v= \nabla\times (v_{0}+V)= \nabla \times V= \nabla\times[\Gamma(V)]= \omega,
\end{equation*}
where in the last equality we have used the first equation in \eqref{div:curl:problem:type1} with $\omega$ solving the transport type system \eqref{transport:problem:type1}-\eqref{omega:o:type1}. Hence,
\begin{equation*}
0=v\cdot\nabla\omega= v\cdot\nabla \left[ \nabla \times v \right]=\nabla \times  \left[  (v\cdot\nabla) v \right], \mbox{ in } \Omega
\end{equation*}
and $\omega_{0}= (\nabla \times v)_{0}$ in $ \partial\Omega_{-}$.
Defining the function $g$ by means of 
\begin{equation}
g=-\int_{0}^{x}(v\cdot \nabla )v (y)\cdot dy,
\end{equation}
we have that
\begin{equation}\label{function:g}
 v\cdot \nabla v=-\nabla g.
\end{equation}
Notice that since  $v=v_{0}+V\in C^{2,\alpha}(\Omega)$, we infer that $g\in C^{2,\alpha}(\Omega)$.
The integral of $g$ is computed along any curve connecting  the origin with $x$. 
We now claim  that $g$ is a uni-valued function in $\Omega=\mathbb{S}^1\times (0,L)$. To this end it suffices that check that $\int_{0}^{1} (v\cdot \nabla) v_{1}  \ dx=0$. Indeed, since $(\nabla \times v)(x,0)=-\frac{\partial V_1}{\partial y}(x,0)+\partial_{x}f$ and
\begin{equation*}
(\nabla \times v)=\partial_{x}f+ \frac{1}{f+v^2_{0}}\left( \partial_{x}h-\partial_{x}(v^{1}_{0}V^1)+\frac{1}{2}\partial_{x}\abs{V^1}^2+f \partial_{y}v^{1}_{0}\right),\  \forall x\in\partial\Omega_{-}
\end{equation*}
we can check that $v=v_0+V$ solves
\begin{equation}\label{value:at:0}
(v\cdot \nabla)v^1=-\partial_{x}(h+p_{0}), \quad  \forall x\in\partial\Omega_{-}.
\end{equation}
Thus,
\begin{equation}
\int_{0}^{1} (v\cdot \nabla) v_{1}  \ dx=- \int_{0}^{1}  \partial_{x}(h+p_{0})  \ dx=0,
\end{equation}
by periodicity in $\mathbb{S}^1$. Hence, $g$ is a uni-valued function. We define the pressure function $p$ as $p=g+g_0$, where $g_0=h(0)+p_{0}(0)$ and then using
 \eqref{function:g} it follows that 
$-\nabla p=(v\cdot\nabla)v$  and  $p(0)=h(0)+p_0(0), \ \forall x\in\partial\Omega_{-}$. Moreover, by equation \eqref{function:g} at $\partial\Omega_{-}$ and \eqref{value:at:0} we observe that
$ \partial_{x} p=\partial_{x}(h+ p_{0})$ and then $p=p_{0}+h$ for $ \forall x\in\partial\Omega_{-}$. Therefore, $p$ solves equation \eqref{Euler2D:eq} with regularity $p\in C^{2,\alpha}(\Omega)$ and  satisfies the boundary condition \eqref{boundary:value:type1}.

On the other hand, let us now assume that $v=v_{0}+V$  with $V\in B_{M}(\Omega)$ is the velocity field which solves $(v,p)$ to \eqref{Euler2D:eq} and   \eqref{boundary:value:type1} with $(v,p)\in C^{2,\alpha}(\Omega)\times C^{2,\alpha}(\Omega)$. Then, taking the curl operator to  \eqref{Euler2D:eq}  we have that
\begin{equation}\label{equation:TP:v_0+V}
(v_{0}+V)\cdot\nabla (\nabla\times V)=0, \ \mbox{in } \ \Omega.
\end{equation} 
Combining \eqref{Euler2D:eq}  and the boundary conditions \eqref{boundary:value:type1}, we obtain that
\begin{equation}\label{equation:TP:v_0+V:0}
(\nabla\times V) = \partial_{x}f+ \frac{1}{f+v^2_{0}}\left( \partial_{x}h-\partial_{x}(v^{1}_{0}V^1)+\frac{1}{2}\partial_{x}\abs{V^1}^2+f \partial_{y}v^{1}_{0}\right), \ \mbox{on} \ \partial \Omega_{-}.
\end{equation}
Therefore $\nabla \times V$ satisfies equation \eqref{transport:problem:type1} with the same boundary condition \eqref{omega:o:type1}. Integrating by characteristics it follows that $\nabla \times V=\omega$. Since $W=\Gamma(V)$ and $V$ satisfy the equations \eqref{div:curl:problem:type1} and Proposition \ref{DCP:prop} implies that the solution is unique, we have that $\Gamma(V)-V=0$ and thus  $V$ is a fixed point of the operator $\Gamma$. 

Given that the fixed point of $\Gamma$ in $B_M(\Omega)$ is unique (i.e. for $\norm{v-v_{0}}_{C^{2,\alpha}(\Omega)}\leq M$), the uniqueness of solutions $(v,p)$ satisfying \eqref{Euler2D:eq} and boundary conditions \eqref{boundary:value:type1} follows.  \quad $\square$
\subsection{Boundary value \textit{(C)}   for the 2D steady Euler equation}\label{S:32}
In this section we will deal with the boundary value \textit{(C)} and construct solutions to the Euler equation \eqref{Euler2D:eq} satisfying those boundary conditions.  However, fixing two arbitrary pressure values on the boundaries $\partial\Omega_{-}$ and $\partial\Omega_{+}$ is not possible in general and some compatibility condition is needed (cf. Theorem \ref{Th:case2:C2} below). Nevertheless, the problem \eqref{Euler2D:eq}with the following  boundary conditions is solvable for a large class of functions $h^{-},h^{+},f^{-}$:
\begin{equation}\label{new:boundary:C}
p=h^{-} \mbox{ on }  \partial\Omega_{-}, \partial_{x}p=\partial_{x}h^{+}  \mbox{ on } \partial\Omega_{+} \mbox{ and } v\cdot n= f^{-}  \mbox{ on }  \partial\Omega_{-}.
\end{equation} 
We then have the following result:
\begin{theorem}\label{Th:case2:C}
 Let $\Omega=\{ (x,y)\in \mathbb{S}^{1}\times (0,L)\}$, with $L>0$ and $\alpha\in (0,1)$.  Suppose that $(v_0,p_0)\in C^{{2,\alpha}}(\Omega)\times C^{{2,\alpha}}(\Omega)$ is a solution of \eqref{Euler2D:eq} with $\bar{v}^{2}_{0}=\displaystyle\inf_{(x,y)\in\bar{\Omega}} \abs{v^2_{0}(x,y)}> 0$ and $\mbox{curl }v_0=0$. For $\mathcal{C}= \{(0,y), y\in[0,L] \}$ we have that the integral $\int_{\mathcal{C}} (v_{0}\cdot n) \ dS$ is a real  constant that we will denote as $J_0$. There exist $\epsilon>0$ , $M>0$ sufficiently small as well as $K>0$ such that for 
$v_0$ as above with $ \norm{v^1_0}_{C^{2,\alpha}(\Omega)}\leq\epsilon $ and  $h^{-}\in C^{{2,\alpha}}(\partial \Omega_{-})$,  $h^{+}\in C^{{2,\alpha}}(\partial \Omega_{+}), f^{-}\in C^{{2,\alpha}}(\partial \Omega_{-})$ and $J\in \mathbb{R}$  satisfying 
\begin{equation}\label{smallnes:bound:C}
\norm{h^-}_{C^{{2,\alpha}}(\partial \Omega_{-})} + \norm{h^+}_{C^{{2,\alpha}}(\partial \Omega_{+})}+\norm{f^{-}}_{C^{{2,\alpha}}(\partial \Omega_{-})}+ \abs{J-J_0} \leq KM ,
\end{equation} 
there exists a unique $(v,p)\in C^{{2,\alpha}}(\Omega)\times C^{{2,\alpha}}(\Omega)$ to \eqref{Euler2D:eq}  with $\norm{v-v_0}_{C^{2,\alpha}(\Omega)}\leq M$ such that
\begin{equation}\label{boundary:value:type1:C}
p =p_{0}+h^- \ \mbox{on} \ \partial \Omega_{-},  \partial_{x}p =\partial_{x}p_{0}+\partial_{x}h^{+} \mbox{on} \ \partial \Omega_{+}, \ v\cdot n=v_{0}\cdot n+f ^{-} \ \mbox{on} \ \partial \Omega_{-} \mbox{ and } \int_{\mathcal{C}}v\cdot n \ dS= J.
\end{equation}
The constants $M,K,$ as well as $\epsilon,$ depend  only on $\alpha,L, \bar{v}^{2}_{0}$.
\end{theorem}
The proof follows the same lines as the proof of Theorem \ref{Th:case2}, hence we will only highlight the main modifications and important points towards the proof.  To that purpose,  let us start by defining the new functional spaces where we will perform the fixed point argument. 

For any,  $M>0$, let us denote by $\hat{B}_{M}$ the closed ball of in $C^{2,\alpha}(\Omega)\times C^{2,\alpha}(\partial\Omega_{+})$ with radius $M$, i.e.,  
\begin{equation}\label{B:M2:alpha:C}
\hat{B}_{M}=\{ g=(g_{1},g_{2})\in C^{2,\alpha}(\Omega)\times C^{2,\alpha}(\partial\Omega_{+}): \norm{g}_{C^{2,\alpha}(\Omega)\times C^{2,\alpha}(\partial\Omega_{+})}\leq M \}.
\end{equation}

We define the operator $\Gamma: \hat{B}_{M} \to C^{2,\alpha}(\Omega)\times C^{2,\alpha}(\partial\Omega_{+})$ in two steps. First, given $(V,f^{+})\in \hat{B}_{M}$ we define $\omega\in C^{1,\alpha}(\Omega)$ solving the following the transport type problem 
\begin{equation}\label{transport:problem:type2}
\left\lbrace
\begin{array}{lll}
(v_{0}+V)\cdot\nabla \omega =0, \ \mbox{in} \ \Omega \\
\omega = \omega_{0}, \ \mbox{on} \ \partial \Omega_{-}
\end{array} \right.
\end{equation} 
with  $\omega_0$ given by
\begin{equation}\label{omega:o:type2}
\omega_{0}=\partial_{x}f^{-}+ \frac{1}{f^{-}+v^2_{0}}\left( -\partial_{x}h^{-}+\partial_{x}(v^{1}_{0}V^1)+\frac{1}{2}\partial_{x}\abs{V^1}^2+f^{-} \partial_{y}v^{1}_{0}\right),\  \forall x\in\partial\Omega_{-},
\end{equation}
where $v_{0}=(v^1_{0},v^2_{0})$ and $V=(V^1,V^2)$.  As a second step, we define $(W,\tilde{f}^{+})\in C^{2,\alpha}(\Omega)\times C^{2,\alpha}(\partial\Omega_{+})$ as
\begin{equation}\label{equation:tilda:f}
\tilde{f}^{+}=\int_{0}^{x} \omega(x',L) +\frac{1}{v^{2}_{0}+f^{+}}\left( \partial_{x}h^{+}  -\partial_{x}(v^{1}_{0}V^1)-\frac{1}{2}\partial_{x}\abs{V^1}^2-f^{+}\partial_{y} v^1_{0} \right)  \ dx' +\tilde{f}^{+}(0,L), \ 
\end{equation}
where the constant $\tilde{f}^{+}(0,L)$ is given by
\begin{equation}\label{constant:f:div}
\tilde{f}^{+}(0,L)=\int_{0}^{1}f^{-}(x) dx- \int_{0}^{1}\mathcal{T}(x)(1-x) \ dx \\
\end{equation}
with $$ \mathcal{T}(x) = \omega +\frac{1}{v^{2}_{0}+f^{+}}\left( \partial_{x}h^{+}  -\partial_{x}(v^{1}_{0}V^1)-\frac{1}{2}\partial_{x}\abs{V^1}^2-f^{+}\partial_{y} v^1_{0} \right).$$ The function $W$ is  the unique solution to the following div-curl problem
 \begin{equation}\label{div:curl:problem:type2}
\left\lbrace
\begin{array}{lll}
\nabla\times W= \omega, \ \mbox{in} \ \Omega \\
\mbox{div } W=0, \ \mbox{in} \ \Omega \\ 
 W\cdot n=\tilde{f}^{+}, \ \mbox{on} \ \partial \Omega_{+} \\ 
 W\cdot n=f^{-}, \ \mbox{on} \ \partial \Omega_{-} \\ 
\displaystyle \int_{\mathcal{C}} W\cdot n \  dS=J-J_{0}.
\end{array} \right.
\end{equation} 
Thus we define $\Gamma(V,f^{+})=(W,\tilde{f}^{+})$.

\subsubsection{Proof of Theorem \ref{Th:case2:C}}
Similarly as in Theorem \ref{Th:case2}, we will show that the operator $\Gamma$ has a fixed point $(V, f^+)$ such that $v=v_{0}+V$ is a solution to \eqref{Euler2D:eq} and \eqref{boundary:value:type1:C}.  To that purpose let $v_0$  be as in Theorem \ref{Th:case2:C}, $h^{-}\in C^{{2,\alpha}}(\partial \Omega_{-})$,  $h^{+}\in C^{{2,\alpha}}(\partial \Omega_{+})$ and $f^{-}\in C^{{2,\alpha}}(\partial \Omega_{-})$ satisfying \eqref{f:condition}. We will show that there exists $\delta_{0}=\delta_{0}(\Omega,\alpha, L, \bar{v}^{2}_{0})$ and $M_{0}\in (0, \frac{\bar{v}^{2}_{0}}{2})$ such that if $M\in (0, M_{0})$ and 
\begin{equation}\label{smallness:case2:C}
\epsilon \leq \delta_{0}, \norm{h^-}_{C^{{2,\alpha}}(\partial \Omega_{-})} + \norm{h^+}_{C^{{2,\alpha}}(\partial \Omega_{+})}+\norm{f^{-}}_{C^{{2,\alpha}}(\partial \Omega_{-})}\leq M_{0}\delta_{0}  \mbox{ and } \abs{J-J_{0}} \leq M_{0}\delta_{0},
\end{equation}
then $\Gamma( \hat{B}_{M}(\Omega))\subset  \hat{B}_{M}(\Omega)$.  Moreover, the operator $\Gamma$ has a unique fixed point in $\hat{B}_{M}(\Omega).$ 

First, let us show that the operator $\Gamma$ maps $ (V,f^{+})\in \hat{B}_{M}$ into itself.  Using inequality \eqref{estimate:DCP} with $j=J-J_{0}$ and then \eqref{estimate:TP}, we have that
\begin{eqnarray}\label{estimate:dcp:2}
\norm{\Gamma(V,f^{+})}_{C^{2,\alpha}}&\leq & C (\norm{\omega}_{C^{1,\alpha}(\Omega)}+\norm{f^{-}}_{C^{2,\alpha}(\partial \Omega_{-})}+\norm{\tilde{f}^{+}}_{C^{2,\alpha}(\partial \Omega_{+})}+\abs{J-J_{0}})  \\
&\leq & C (\norm{\omega_0}_{C^{1,\alpha}(\partial\Omega_{-})}+\norm{f^-}_{C^{2,\alpha}(\partial \Omega_{-})}+\norm{\tilde{f}^{+}}_{C^{2,\alpha}(\partial \Omega_{+})}+\abs{J-J_{0}}) \nonumber
\end{eqnarray}

On the one hand, since $(V,f^{+})\in \hat{B}_{M}(\Omega)$ and \eqref{smallness:case2:C}  the function $\omega_{0}$ defined in  \eqref{omega:o:type2} can be estimated as in \eqref{estimate:omega:1}
\begin{eqnarray}\label{estimate:omega:2}
\norm{\omega_{0}}_{C^{1,\alpha}(\partial\Omega_{-})}\ &=&\norm{\partial_{x}f^{-}+ \frac{1}{f^{-}+v^2_{0}}\left( \partial_{x}h^{-}-\partial_{x}(v^{1}_{0}V^1)+\frac{1}{2}\partial_{x}\abs{V^1}^2+f^{-} \partial_{y}v^{1}_{0}\right)}_{C^{1,\alpha}(\partial\Omega_{-})} \nonumber \\ 
 &\leq & C(\delta_0M+M^2).
\end{eqnarray}
On the other hand,  using \eqref{equation:tilda:f}, \eqref{estimate:TP} and bound \eqref{estimate:omega:2} we have that
\begin{eqnarray}\label{bound:tild:f}
\norm{\tilde{f^{+}}}_{C^{2,\alpha}(\partial\Omega_{+})} &\leq & C\bigl(\norm{\omega_{0}}_{C^{1,\alpha}(\partial\Omega_{-})}+ \norm{h^+}_{C^{2,\alpha}(\partial\Omega_{+})} + \epsilon \norm{V}_{C^{2,\alpha}(\Omega)} +  \norm{V}^2_{C^{2,\alpha}(\Omega)} \nonumber \\
&& \ + \epsilon \norm{f^{+}}_{C^{2,\alpha}(\partial \Omega_{+})}+\norm{f^-}_{C^{2,\alpha}}  \bigr) \leq C\left(\delta_0M+M^2 \right).
\end{eqnarray}
Combining estimates \eqref{estimate:dcp:2}-\eqref{bound:tild:f} we have that
\begin{eqnarray}
\norm{\Gamma(V,f^{+})}_{C^{2,\alpha}} &\leq &  C\left(\delta_0M+M^2+ \norm{f^{-}}_{C^{2,\alpha}(\partial \Omega_{-})}+\abs{J-J_{0}} \right) \nonumber \\
&\leq & C\left(\delta_0M+M^2 \right),
\end{eqnarray}
where we have used the smallness assumption \eqref{smallness:case2:C} in the last inequality.
Choosing $C\delta_{0}\leq\frac{1}{4}$ and $CM_{0}\leq \frac{1}{4}$, we obtain $\Gamma( \hat{B}_{M}(\Omega))\subset  \hat{B}_{M}(\Omega)$, for $M \in (0,M_{0})$.

By mimicking the arguments of Lemma \ref{Lemma:contraction},  we can show that $\hat{B}_{M}(\Omega)$ endowed with the topology $C^{1,\alpha}\times C^{1,\alpha}$ is a complete metric space denoted by $(\hat{B}_{M}, \norm{\cdot}_{C^{1,\alpha}(\Omega)\times C^{1,\alpha}(\partial\Omega_{+})})$.  We claim that $\Gamma: (\hat{B}_{M}, \norm{\cdot}_{C^{1,\alpha}(\Omega)\times C^{1,\alpha}(\partial\Omega_{+})}) \to (\hat{B}_{M}, \norm{\cdot}_{C^{1,\alpha}(\Omega)\times C^{1,\alpha}(\partial\Omega_{+})})$ is a contraction mapping. 

To that purpose, we need to estimate the difference $\norm{\Gamma(V^{1}, f^{1,+})-\Gamma(V^{2},f^{2,+})}_{C^{1,\alpha}}$ for $(V^{1}, f^{1,+}),(V^{2},f^{2,+})\in  \hat{B}_{M}(\Omega)$. Using the linearity of problem \eqref{div:curl:problem:type2},  we have that bound \eqref{estimate:DCP} yields
\begin{equation}\label{difference:bound:2}
\norm{\Gamma(V^{1}, f^{1,+})-\Gamma(V^{2},f^{2,+})}_{C^{1,\alpha}(\Omega)} \leq C \norm{\omega^1-\omega^2}_{C^{\alpha}(\Omega)} + \norm{\tilde{f}^{1,+}-\tilde{f}^{2,+}}_{C^{1,\alpha}(\partial\Omega_{+})}.
\end{equation}
Using \eqref{estimate:difference:0}, \eqref{omega:0:1} and the smallness assumption \eqref{smallness:case2:C} we infer that the difference $\omega^1-\omega^2$ can be bounded by
\begin{equation}\label{difference:bound:4}
\norm{\omega^1-\omega^2}_{C^{\alpha}(\Omega)} \leq C \left(\delta^2_{0}M+\delta_{0}M+M^2\right)\norm{V^1-V^2}_{C^{1,\alpha}(\Omega)}.
\end{equation}

To estimate the latter term on the right hand side in \eqref{difference:bound:2}, we notice that using \eqref{equation:tilda:f}  
\begin{equation*}
\norm{\tilde{f}^{1,+}-\tilde{f}^{2,+}}_{C^{1,\alpha}(\partial\Omega_{+})}= I_{1}+I_{2}
\end{equation*}
with 
\begin{eqnarray*}
I_1=  \norm{\displaystyle\int_{0}^{x}\omega^{1}(s,L)-\omega^2(s,L) \ ds}_{C^{1,\alpha}(\partial\Omega_{+})} &\leq& \norm{\omega^{1}-\omega^{2}}_{C^{\alpha}(\partial\Omega_{+})} \\
&\leq & C \left(\delta^2_{0}M+\delta_{0}M+M^2\right)\norm{V^1-V^2}_{C^{1,\alpha}(\Omega)}
\end{eqnarray*}
and 
$$
I_2 = \norm{\displaystyle\int_{0}^{x} \frac{1}{(v^{2}_{0}+f^{1,+})-(v^{2}_{0}+f^{2,+})} \mathcal{M}(s)\ ds}_{C^{1,\alpha}(\partial\Omega_{+})} \leq C\norm{\mathcal{M}}_{C^{\alpha}(\partial \Omega_{+} )}.
$$
To bound the last term $I_2$, we have used estimate \eqref{estimate:phi:h} with $\mathcal{H}(x)= \frac{1}{(v^{2}_{0}+f^{1,+})-(v^{2}_{0}+f^{2,+})} \mathcal{M}(x)$ where
$$\mathcal{M}(x)= \left(-\partial_{x}(v^{1}_{0}(x)(V^{1}_{1}-V^{2}_{1}))-V^{1}_{2}\partial_{x}(V^{1}_{1}-V^{2}_{1})-(V^{1}_{1}-V^{2}_{1})\partial_{x}V^{1}_{1}-\partial_{y}v^{1}_{0}(f^{1,+}-f^{2,+})\right).$$
Hence, by the smallness assumption \eqref{smallness:case2:C}
\begin{equation}\label{difference:bound:6}
\norm{\tilde{f}^{1,+}-\tilde{f}^{2,+}}_{C^{1,\alpha}(\partial\Omega_{+})} \leq C \left((\delta_{0}M+M^2)\norm{V^1-V^2}_{C^{1,\alpha}(\Omega)}+\delta_0 M_{0} \norm{f^{1,+}-f^{2,+}}_{C^{1,\alpha}(\partial\Omega_{+})}\right).
\end{equation}
Therefore,  collecting \eqref{difference:bound:4}-\eqref{difference:bound:6} yields
\begin{eqnarray}\label{difference:bound:7}
\norm{\Gamma(V^{1}, f^{1,+})-\Gamma(V^{2},f^{2,+})}_{C^{1,\alpha}(\Omega)} &\leq & C \left((\delta_{0}M_{0}+M_{0}^2)\norm{V^1-V^2}_{C^{1,\alpha}(\Omega)}+\delta_0 M_{0} \norm{f^{1,+}-f^{2,+}}_{C^{1,\alpha}(\partial\Omega_{+})}\right) \nonumber \\
&< & \theta(\norm{V^1-V^2}_{C^{1,\alpha}}+\norm{f^{1,+}-f^{2,+}}_{C^{1,\alpha}}) 
\end{eqnarray}
where $\theta$ is strictly less than one for  $C\delta_{0}\leq\frac{1}{4}$ and $CM_{0}\leq \frac{1}{4}$ that  $\Gamma: (\hat{B}_{M}, \norm{\cdot}_{C^{1,\alpha}(\Omega)\times C^{1,\alpha}(\partial\Omega_{+})}) \to (\hat{B}_{M}, \norm{\cdot}_{C^{1,\alpha}(\Omega)\times C^{1,\alpha}(\partial\Omega_{+})})$ is a contraction mapping. Invoking Banach's fixed point theorem we can conclude that $\Gamma$ admits a unique fixed point 	$(V,f^{+})\in \hat{B}_{M}(\Omega)$, i.e. $\Gamma(V,f^{+})=(V,f^{+})$, which concludes the proof.

To conclude the proof of Theorem \ref{Th:case2:C}, we need to check that $(V,f^{+})\in \hat{B}_{M}(\Omega)$ is a fixed point of the operator $\Gamma$ if an only if $v=v_0+V$ is the velocity field which is a solution $(v,p)\in C^{{2,\alpha}}(\Omega)\times C^{{2,\alpha}}(\Omega)$ to \eqref{Euler2D:eq} and \eqref{boundary:value:type1:C}.

Assuming that $(V,f^{+})\in \hat{B}_{M}(\Omega)$ is a fixed point of the operator $\Gamma$ we can conclude by repeating the arguments of the proof of Theorem \ref{Th:case2} that 
\[\mbox{div } v=0, \quad \mbox{in } \Omega, \quad v\cdot n =  v_{0}\cdot n+ f^{-}, \mbox{in } \partial \Omega_{-}, \mbox{and }  \int_{\mathcal{C}}v\cdot n \ dS=J.\]
Moreover, $v$ satisfies the vorticity formulation of Euler equation
\begin{equation*}
0=v\cdot\nabla\omega= v\cdot\nabla \left[ \nabla \times v \right]=\nabla \times  \left[  (v\cdot\nabla) v \right], \mbox{ in } \Omega
\end{equation*}
and $\omega_{0}= (\nabla \times v)_{0}$ in $ \partial\Omega_{-}$. The difference with respect to Theorem \ref{Th:case2} relies on how to construct a pressure field $p$ which satisfies the boundary conditions \eqref{boundary:value:type1:C}. By similar arguments as the ones in Theorem \ref{Th:case2}, defining a uni-valued function $g$ in $\Omega=\mathbb{S}^1\times (0,L)$ as 
\begin{equation}\label{function:g:2:integral}
g=-\int_{0}^{x}(v\cdot \nabla )v (y)\cdot dy,
\end{equation}
we have that
\begin{equation}\label{function:g:2}
 v\cdot \nabla v=-\nabla g,
\end{equation}
with $g\in C^{2,\alpha}(\Omega)$. Defining $p=g+g_{0}$ with  $g_{0}=h^{-}(0)+p_{0}(0)$ we can check using \eqref{function:g:2:integral}-\eqref{function:g:2} that $p=p_{0}+h^{-}$ for $x\in \partial\Omega_-$ and $\partial_{x}p=\partial_{x}p_{0}+\partial_{x}h^{+}$ for $x\in \partial\Omega_+$. Hence, $p$ solves equation \eqref{Euler2D:eq} with regularity $p\in C^{2,\alpha}(\Omega)$ and satisfies the boundary condition \eqref{boundary:value:type1:C}. $\square$ \\

In order to solve the boundary value problem \textit{(C)} for the Euler equation as stated in Table \ref{tab:table1} (i.e a boundary problem for $p \text{ on }\partial \Omega, v\cdot n \text{ on } \partial\Omega_{-}$ and the flux $J$) we need to impose a compatibility condition on the boundary values $h^{-}, h^{+}, f{-}$ as well as in the flux $J$. To that purpose, we define a subset $\mathcal{S}_{KM}\subset C^{2,\alpha}(\Omega_{-}) \times C^{2,\alpha}(\Omega_{+})\times C^{2,\alpha}(\Omega_{-})\times \mathbb{R}$ given by 
\begin{equation}
\mathcal{S}_{KM}=\{ (h^{-},h^{+},f^{-},J) \text{ satisfying } \eqref{smallnes:bound:C} \},
\end{equation}
and the operator $\Lambda:\mathcal{S}_{K,M}\to \mathbb{R}$ given by $\Lambda(h^{-},h^{+},f^{-},J) = p(0,L)-p_{0}(0,L)$ where $p$ is the pressure function obtained in Theorem \ref{Th:case2:C}. Notice that by construction $\Lambda(h^{-},h^{+}+a,f^{-},J)=\Lambda(h^{-},h^{+},f^{-},J)$ for any real constant $a$. The definition of the operator $\Lambda$ suggest that the following compatibility condition is required to solve the boundary value problem (\textit{C}) 
\begin{equation}\label{compatibility}
h^{+}(0)= \Lambda(h^{-},h^{+},f^{-},J).
\end{equation}
More precisely the following theorem holds:
\begin{theorem}\label{Th:case2:C2}
 Let $\Omega=\{ (x,y)\in \mathbb{S}^{1}\times (0,L)\}$, with $L>0$ and $\alpha\in (0,1)$.  Suppose that $(v_0,p_0)\in C^{{2,\alpha}}(\Omega)\times C^{{2,\alpha}}(\Omega)$ is a solution of \eqref{Euler2D:eq} with $\bar{v}^{2}_{0}=\displaystyle\inf_{(x,y)\in\bar{\Omega}} \abs{v^2_{0}(x,y)}> 0$ and  $\mbox{curl }v_0=0$. For $\mathcal{C}= \{(0,y), y\in[0,L] \}$ we have that the integral $\int_{\mathcal{C}} (v_{0}\cdot n) \ dS$ is a real  constant that we will denote as $J_0$. There exist $\epsilon>0$ , $M>0$ sufficiently small as well as $K>0$ such that for 
$v_0$ as above with $ \norm{v^1_0}_{C^{2,\alpha}(\Omega)}\leq\epsilon $ and  $h^{-}\in C^{{2,\alpha}}(\partial \Omega_{-})$,  $h^{+}\in C^{{2,\alpha}}(\partial \Omega_{+}), f^{-}\in C^{{2,\alpha}}(\partial \Omega_{-})$ and $J\in \mathbb{R}$  satisfying 
\begin{equation}
\norm{h^-}_{C^{{2,\alpha}}(\partial \Omega_{-})} + \norm{h^+}_{C^{{2,\alpha}}(\partial \Omega_{+})}+\norm{f^{-}}_{C^{{2,\alpha}}(\partial \Omega_{-})}+ \abs{J-J_0} \leq KM ,
\end{equation} 
there exists a unique $(v,p)\in C^{{2,\alpha}}(\Omega)\times C^{{2,\alpha}}(\Omega)$ to \eqref{Euler2D:eq}  with $\norm{v-v_0}_{C^{2,\alpha}(\Omega)}\leq M$ such that
\begin{equation}\label{boundary:value:type1:C2}
p =p_{0}+h^- \ \mbox{on} \ \partial \Omega_{-},  p =p_{0}+h^{+} \mbox{on} \ \partial \Omega_{+}, \ v\cdot n=v_{0}\cdot n+f ^{-} \ \mbox{on} \ \partial \Omega_{-} \mbox{ and } \int_{\mathcal{C}}v\cdot n \ dS= J
\end{equation}
if and only if 
\begin{equation}\label{compatibility2}
h^{+}(0)= \Lambda(h^{-},h^{+},f^{-},J).
\end{equation}
\end{theorem}
\begin{proof}[Proof of Theorem \ref{Th:case2:C2}]
By Theorem \ref{Th:case2:C}, we have that that there exists a unique solution $(v,p)\in C^{{2,\alpha}}(\Omega)\times C^{{2,\alpha}}(\Omega)$ to \eqref{Euler2D:eq} satisfying the boundary conditions \eqref{boundary:value:type1:C}, this is 
$$ p =p_{0}+h^- \ \mbox{on} \ \partial \Omega_{-},  \partial_{x}p =\partial_{x}p_{0}+\partial_{x}h^{+} \mbox{on} \ \partial \Omega_{+}, \ v\cdot n=v_{0}\cdot n+f ^{-} \ \mbox{on} \ \partial \Omega_{-} \mbox{ and } \int_{\mathcal{C}}v\cdot n \ dS= J.$$
Integrating the pressure boundary condition at $\partial \Omega_{+}$ yields
\begin{eqnarray*}
p(x,L)&=& p(0,L)+p_{0}(x,L)-p_{0}(x,L) + \int_{0}^{x} \partial_{x}h^{+}(\xi) d\xi  \\
&=&p_{0}(x,L)+h^{+}(x,L)+\Lambda(h^{-},h^{+},f^{-},J)-h^{+}(0,L).
\end{eqnarray*}
Then the problem  \eqref{Euler2D:eq} with boundary conditions \eqref{boundary:value:type1:C2} has a unique solution if and only if 
$h^{+}(0)= \Lambda(h^{-},h^{+},f^{-},J).$
\end{proof}

\subsection{Boundary value \textit{(G)} for the 2D steady Euler equation}\label{S:3:3}
In this section we will sketch the proof regarding the construction of solutions to the Euler equation \eqref{Euler2D:eq} satisfying boundary conditions \textit{(G)}, which is a slight modification of the proof of Theorem \ref{Th:case2} and Theorem \ref{Th:case2:C}. 

\begin{theorem}\label{Theorem:G:Euler}
 Let $\Omega=\{ (x,y)\in \mathbb{S}^{1}\times (0,L)\}$, with $L>0$ and $\alpha\in (0,1)$.  Suppose that $(v_0,p_0)\in C^{{2,\alpha}}(\Omega)\times C^{{2,\alpha}}(\Omega)$ is a solution of \eqref{Euler2D:eq} with $\bar{v}^{2}_{0}=\displaystyle\inf_{(x,y)\in\bar{\Omega}} \abs{v^2_{0}(x,y)}> 0$ and  $\mbox{curl }v_0=0$. For $\mathcal{C}= \{(0,y), y\in[0,L] \}$ we have that the integral $\int_{\mathcal{C}} (v_{0}\cdot n) \ dS$ is a real  constant that we will denote as $J_0$. There exist $\epsilon>0$ , $M>0$ sufficiently small as well as $K>0$ such that for 
$v_0$ as above with $ \norm{v^1_0}_{C^{2,\alpha}(\Omega)}\leq\epsilon $ and  $h^{-}\in C^{{2,\alpha}}(\partial \Omega_{-})$,  $h^{+}\in C^{{2,\alpha}}(\partial \Omega_{+}), f^{-}\in C^{{2,\alpha}}(\partial \Omega_{-})$ and $J\in \mathbb{R}$  satisfying 
\begin{equation}\label{smallnes:bound:G}
\norm{h^-}_{C^{{2,\alpha}}(\partial \Omega_{-})} + \norm{h^+}_{C^{{2,\alpha}}(\partial \Omega_{+})}+\norm{f^{-}}_{C^{{2,\alpha}}(\partial \Omega_{-})}+ \abs{J-J_0} \leq KM ,
\end{equation} 
there exists a unique $(v,p)\in C^{{2,\alpha}}(\Omega)\times C^{{2,\alpha}}(\Omega)$ to \eqref{Euler2D:eq}  with $\norm{v-v_0}_{C^{2,\alpha}(\Omega)}\leq M$ such that
\begin{equation}
p +\frac{\abs{v}^2}{2}=p_{0}+\frac{\abs{v_{0}}^2}{2}+h^- \ \mbox{on} \ \partial \Omega_{-},  p =p_{0}+h^{+} \mbox{on} \ \partial \Omega_{+}, \ v\cdot n=v_{0}\cdot n+f ^{-} \ \mbox{on} \ \partial \Omega_{-} \mbox{ and } \int_{\mathcal{C}}v\cdot n \ dS= J.
\end{equation}
The constants $M,K,$ as well as $\epsilon,$ depend  only on $\alpha,L, \bar{v}^{2}_{0}$.
\end{theorem}

\begin{proof}[Proof of Theorem \ref{Theorem:G:Euler}]
For any,  $M>0$, let us denote by $\hat{B}_{M}$ the closed ball of in $C^{2,\alpha}(\Omega)\times C^{2,\alpha}(\partial\Omega_{+})$ with radius $M$, i.e.,  
\begin{equation}
\hat{B}_{M}=\{ (g_{1},g_{2})\in C^{2,\alpha}(\Omega)\times C^{2,\alpha}(\partial\Omega_{-}): \norm{g}_{C^{2,\alpha}(\Omega)\times C^{2,\alpha}(\partial\Omega_{+})}\leq M \}.
\end{equation}

We define the operator $\Gamma: \hat{B}_{M} \to C^{2,\alpha}(\Omega)\times C^{2,\alpha}(\partial\Omega_{+})$ in two steps. First, given $(V,f^{+})\in \hat{B}_{M}$ we define $\omega\in C^{1,\alpha}(\Omega)$ solving the following the transport type problem 
\begin{equation}\label{transport:problem:type3}
\left\lbrace
\begin{array}{lll}
(v_{0}+V)\cdot\nabla \omega =0, \ \mbox{in} \ \Omega \\
\omega = \omega_{0}, \ \mbox{on} \ \partial \Omega_{-}
\end{array} \right.
\end{equation} 
with  $\omega_0$ given by
\begin{equation}\label{omega:o:type3}
\omega_{0}=\frac{\partial_{x}h^{-}}{v^2_0+f^{-}}, \  \forall x\in\partial\Omega_{-},
\end{equation}
where $v_{0}=(v^1_{0},v^2_{0})$ and $V=(V^1,V^2)$.  As a second step, we define $(W,\tilde{f}^{+})\in C^{2,\alpha}(\Omega)\times C^{2,\alpha}(\partial\Omega_{+})$ as
\begin{equation}\label{equation:tilda:f:G}
\tilde{f}^{+}=\int_{0}^{x} \omega(x',L) +\frac{1}{v^{2}_{0}+f^{+}}\left( h^{+}  -\partial_{x}(v^{1}_{0}V^1)-\frac{1}{2}\partial_{x}\abs{V^1}^2-f^{+}\partial_{y} v^1_{0} \right)  \ dx' +\tilde{f}^{+}(0,L), \ 
\end{equation}
where the constant $\tilde{f}^{+}(0,L)$ is given by
\begin{equation}\label{constant:f:div:G}
\tilde{f}^{+}(0,L)=\int_{0}^{1}f^{-}(x) dx- \int_{0}^{1}\mathcal{T}(x)(1-x) \ dx \\
\end{equation}
with $$ \mathcal{T}(x)= \omega +\frac{1}{v^{2}_{0}+f^{+}}\left( h^{+}  -\partial_{x}(v^{1}_{0}V^1)-\frac{1}{2}\partial_{x}\abs{V^1}^2-f^{+}\partial_{y} v^1_{0} \right).$$  The function $W$ is the unique solution to the following div-curl problem
 \begin{equation}\label{div:curl:problem:type3}
\left\lbrace
\begin{array}{lll}
\nabla\times W= \omega, \ \mbox{in} \ \Omega \\
\mbox{div } W=0, \ \mbox{in} \ \Omega \\ 
 W\cdot n=\tilde{f}^{+}, \ \mbox{on} \ \partial \Omega_{+} \\ 
 W\cdot n=f^{-}, \ \mbox{on} \ \partial \Omega_{-} \\ 
\displaystyle \int_{\mathcal{C}} W\cdot n \  dS=J-J_{0}.
\end{array} \right.
\end{equation} 
Thus we define $\Gamma(V,f^{+})=(W,\tilde{f}^{+})$. We need to show that there exists a sufficiently small $\delta_{0}=\delta_{0}(\Omega,\alpha,L, \bar{v}^2_{0})$ and $M_{0}\in (0, \frac{ \bar{v}^2_{0}}{2})$ such that if $M\in (0, M_{0})$ and 
\begin{equation}\label{smallness:case2:G}
\epsilon \leq \delta_{0}, \norm{h^-}_{C^{{2,\alpha}}(\partial \Omega_{-})} + \norm{h^+}_{C^{{2,\alpha}}(\partial \Omega_{+})}+\norm{f^{-}}_{C^{{2,\alpha}}(\partial \Omega_{-})}\leq M_{0}\delta_{0}  \mbox{ and } \abs{J-J_{0}} \leq M_{0}\delta_{0},
\end{equation}
then $\Gamma( \hat{B}_{M}(\Omega))\subset  \hat{B}_{M}(\Omega)$.  Similarly, as for case \textit{(C)} we have that using inequalities \eqref{estimate:DCP} with $j=J-J_{0}$ and \eqref{estimate:TP} we have that
\begin{equation}
\norm{\Gamma(V,f^+)}_{C^{2,\alpha}(\Omega)} \leq C \left(\norm{\omega_0}_{C^{1,\alpha}(\partial\Omega_{-})}+\norm{f^-}_{C^{2,\alpha}(\partial \Omega_{-})}+\norm{\tilde{f}^{+}}_{C^{2,\alpha}(\partial \Omega_{+})}+\abs{J-J_{0}}\right).
\end{equation}
Furthermore, by the definition of $\omega_{0}$ and $\tilde{f^+}$ in \eqref{omega:o:type3}, \eqref{equation:tilda:f:G}respectively and using the smallness assumption \eqref{smallnes:bound:G} we have that
\begin{equation}\label{estimate:omega:G}
\norm{\omega_{0}}_{C^{1,\alpha}(\Omega)}\leq C\delta_{0}M,
\end{equation}
and 
\begin{equation}\label{estimate:tild:f:G}
\norm{\tilde{f^{+}}}_{C^{2,\alpha}(\partial\Omega_{+})} \leq  C\left(\delta_0M+M^2 \right).
\end{equation}
Hence, using \eqref{estimate:omega:G} and \eqref{estimate:tild:f:G} we have that
$$ \norm{\Gamma(V,f^{+})}_{C^{2,\alpha}(\Omega)} \leq C\left(\delta_0 M+M^2 \right). $$
Choosing $C \delta_0 \leq \frac{1}{4}$ and $CM_0 \leq \frac{1}{4}$, we obtain $\Gamma( \hat{B}_{M}(\Omega))\subset  \hat{B}_{M}(\Omega)$, for $M \in (0,M_{0})$. 

We claim that $\Gamma: (\hat{B}_{M}, \norm{\cdot}_{C^{1,\alpha}(\Omega)\times C^{1,\alpha}(\partial\Omega_{+})}) \to (\hat{B}_{M}, \norm{\cdot}_{C^{1,\alpha}(\Omega)\times C^{1,\alpha}(\partial\Omega_{+})})$ is a contraction mapping. To that purpose, we estimate the difference 
$\norm{\Gamma(V^{1}, f^{1,+})-\Gamma(V^{2},f^{2,+})}_{C^{1,\alpha}}$ with $(V^{1}, f^{1,+}),(V^{2},f^{2,+})\in  \hat{B}_{M}(\Omega)$. Mimicking the estimates \eqref{difference:bound:2}, \eqref{difference:bound:4} and \eqref{difference:bound:6} it follows that  that
$$ \norm{\Gamma(V^{1}, f^{1,+})-\Gamma(V^{2},f^{2,+})}_{C^{1,\alpha}(\Omega)} \leq  \theta(\norm{V^1-V^2}_{C^{1,\alpha}}+\norm{f^{1,+}-f^{2,+}}_{C^{1,\alpha}}),$$
where $\theta$ is strictly less than one for  $C\delta_{0}\leq\frac{1}{4}$ and $CM_{0}\leq \frac{1}{4}$, and hence showing that  $\Gamma: (\hat{B}_{M}, \norm{\cdot}_{C^{1,\alpha}(\Omega)\times C^{1,\alpha}(\partial\Omega_{+})}) \to (\hat{B}_{M}, \norm{\cdot}_{C^{1,\alpha}(\Omega)\times C^{1,\alpha}(\partial\Omega_{+})})$ is a contraction mapping. We conclude the proof by using Banach's fixed point theorem which shows the existence of a unique fixed point  $(V,f^{+})\in \hat{B}_{M}(\Omega)$. 
\end{proof}

\subsection{Loss of regularity for the boundary value problem \textit{(D)} using vorticity transport method}\label{S:3:4}
In this section we provide a justification to the regularity loss of the vorticity transport method to deal with boundary value problem \textit{(D)} (i.e. a boundary problem for $p+\frac{v^2}{2}=h^{-}$ on $\partial \Omega_{-}$, $p+\frac{v^2}{2}=h^{+}$ on $\partial \Omega_{+}$ and $v\cdot n$ on $\partial\Omega_{-}$). 

Suppose that we try to solve this problem using the vorticity transport method used before to deal with the cases \textit{(B),(C)}.
In a similar way as above, we define the operator $\Gamma: \hat{B}_{M} \to C^{2,\alpha}(\Omega)\times C^{2,\alpha}(\partial\Omega_{+})$ in two steps. First, given $(V,f^{+})\in \hat{B}_{M}$ we define $\omega\in C^{1,\alpha}(\Omega)$ solving the following the transport type problem 
\begin{equation}\label{transport:problem:type4}
\left\lbrace
\begin{array}{lll}
(v_{0}+V)\cdot\nabla \omega =0, \ \mbox{in} \ \Omega \\
\omega = \omega_{0}, \ \mbox{on} \ \partial \Omega_{-}
\end{array} \right.
\end{equation} 
with  $\omega_0$ given by
\begin{equation}
\omega_{0}=\frac{\partial_{x}h^{-}}{v^2_0+f^{-}}, \  \forall x\in\partial\Omega_{-},
\end{equation}
where $v_{0}=(v^1_{0},v^2_{0})$ and $V=(V^1,V^2)$.  As a second step we define $(W,\tilde{f}^{+})$ where
$W$ is the unique solution to the following div-curl problem
 \begin{equation}\label{div:curl:problem:type4}
\left\lbrace
\begin{array}{lll}
\nabla\times W= \omega, \ \mbox{at} \ \Omega \\
\mbox{div } W=0, \ \mbox{at} \ \Omega \\ 
 W\cdot n=\tilde{f}^{+}, \ \mbox{at} \ \partial \Omega_{+} \\ 
 W\cdot n=f^{-}, \ \mbox{at} \ \partial \Omega_{-} \\ 
\displaystyle \int_{\mathcal{C}} W\cdot n \  dS=J-J_{0},
\end{array} \right.
\end{equation} 
and the function $\tilde{f}^{+}$ is given by 
\begin{equation}
\tilde{f}^{+}=\frac{\partial_{x}h^{+}}{\omega}, \  \forall x\in\partial\Omega_{+},
\end{equation}
Thus we define $\Gamma(V,f^{+})=(W,\tilde{f}^{+})$. Therefore, since $(V,f^{+})\in C^{2,\alpha}(\Omega)\times C^{2,\alpha}(\partial\Omega_{+})$ and $\omega_{0}\in C^{1,\alpha}(\Omega_{-})$ we infer that $\omega\in C^{1,\alpha}(\Omega)$. Furthermore, the function $\tilde{f}^{+}$ has regularity $C^{1,\alpha}(\Omega_{+})$ and and the function $W$ can only to be expected to have at most regulartity $C^{1,\alpha}$. Therefore it would not be possible to close the fixed point argument.

\section{The MHS boundary value theorems}\label{S:4}
For the sake of completeness, we will present in this section the statements of the theorems we have shown for the boundary value problems for the Euler equation \eqref{Euler2D:eq:omega} in the case of the MHS equations \eqref{MHS2D:current}. Since from the PDE point the Euler and MHS equation are equivalent problems using the identification variables \eqref{transformation:variable} the proof is exactly the same as in the Euler case. 

\begin{theorem}[Case \textit{A}]\label{theorem:A:MHS}
Let $f^{-}\in C^{1,\alpha}(\partial\Omega_{-}), f^{+}\in C^{1,\alpha}(\partial\Omega_{+})$ and $h^{-}\in C^{1,\alpha}(\partial\Omega_{+})$. Then if $f^{-}>0$ for $x\in \partial \Omega_{-}$, there exists a solution $(B,p)\in C^{{1,\alpha}}(\Omega)\times C^{{1,\alpha}}(\Omega)$ solving the MHS equation \eqref{MHS2D:current} such that 
\begin{equation}
B\cdot n =f^{-}\mbox{on} \ \partial \Omega^{-}, B\cdot n =f^{+} \ \mbox{on} \ \partial \Omega_{+} \mbox{ and }  p=h^{-}  \ \mbox{on} \ \partial \Omega_{-}.
\end{equation}
Moreover, there exists $\delta>0$ such that if 
\begin{equation}
\norm{f^-}_{C^{{1,\alpha}}(\partial \Omega_{-})}+\norm{f^+}_{C^{{1,\alpha}}(\partial \Omega_{+})} +\norm{h^{-}}_{C^{{1,\alpha}}(\partial \Omega_{-})}\leq  \delta ,
\end{equation}
the solution $(B,p)$ is unique. 
\end{theorem}

\begin{theorem}[Case \textit{B}]\label{theorem:B:MHS} 
Let $\Omega=\{ (x,y)\in \mathbb{S}^{1}\times (0,L)\}$, with $L>0$ and $\alpha\in (0,1)$.  Suppose that $(B_0,p_0)\in C^{{2,\alpha}}(\Omega)\times C^{{2,\alpha}}(\Omega)$ is a solution of \eqref{MHS2D:current} with $\bar{B}^{2}_{0}=\displaystyle\inf_{(x,y)\in\bar{\Omega}} \abs{B^2_{0}(x,y)}> 0$  and  $\mbox{curl }B_0=0$. For $\mathcal{C}= \{(0,y), y\in[0,L] \}$ we have that the integral $\int_{\mathcal{C}} (B_{0}\cdot n) \ dS$ is a real  constant that we will denote as $J_0$. There exist $\epsilon>0$ , $M>0$ sufficiently small as well as $K>0$ such that for 
$B_0$ as above with $ \norm{B^1_0}_{C^{2,\alpha}(\Omega)}\leq\epsilon $ and  $h\in C^{{2,\alpha}}(\partial \Omega_{-})$, $f\in C^{{2,\alpha}}(\partial \Omega)$ and $J\in \mathbb{R}$  satisfying 
\begin{equation}
\norm{h}_{C^{{2,\alpha}}(\partial \Omega_{-})}+\norm{f}_{C^{{2,\alpha}}(\partial \Omega)}+ \abs{J-J_0} \leq KM ,
\end{equation} 
and 
\begin{equation}
\int_{\partial\Omega_{-}}f \ dS=  \int_{\partial\Omega_{+}}f \ dS,
\end{equation}
there exists a unique $(B,p)\in C^{{2,\alpha}}(\Omega)\times C^{{2,\alpha}}(\Omega)$ to \eqref{MHS2D:current}  with $\norm{B-B_0}_{C^{2,\alpha}(\Omega)}\leq M$ such that
\begin{equation}
B\cdot n =B_{0}\cdot n+f \ \mbox{on} \ \partial \Omega, \ p+\frac{\abs{B}^2}{2}=p_0+\frac{\abs{B_0}^2}{2}+h  \ \mbox{on} \ \partial \Omega_{-} \mbox{ and } \int_{\mathcal{C}}B\cdot n \ dS= J.
\end{equation}
The constants $M,K,$ as well as $\epsilon,$ depend  only on $\alpha,L, \bar{B}^{2}_{0}$.
\end{theorem}

\begin{theorem}[Case \textit{C}. Solvability]\label{theorem:C:MHS} 
Let $\Omega=\{ (x,y)\in \mathbb{S}^{1}\times (0,L)\}$, with $L>0$ and $\alpha\in (0,1)$.  Suppose that $(B_0,p_0)\in C^{{2,\alpha}}(\Omega)\times C^{{2,\alpha}}(\Omega)$ is a solution of \eqref{MHS2D:current} with $\bar{B}^{2}_{0}=\displaystyle\inf_{(x,y)\in\bar{\Omega}} \abs{B^2_{0}(x,y)}> 0$ and $\mbox{curl }B_0=0$.. For $\mathcal{C}= \{(0,y), y\in[0,L] \}$ we have that the integral $\int_{\mathcal{C}} (B_{0}\cdot n) \ dS$ is a real  constant that we will denote as $J_0$. There exist $\epsilon>0$ , $M>0$ sufficiently small as well as $K>0$ such that for 
$B_0$ as above with $ \norm{B^1_0}_{C^{2,\alpha}(\Omega)}\leq\epsilon $ and $h^{-}\in C^{{2,\alpha}}(\partial \Omega_{-})$, $h^{+}\in C^{{2,\alpha}}(\partial \Omega_{+})$, $f^{-}\in C^{{2,\alpha}}(\partial \Omega_{-})$ and $J\in \mathbb{R}$  satisfying 
\begin{equation}
\norm{h^{-}}_{C^{{2,\alpha}}(\partial \Omega_{-})}+\norm{h^{+}}_{C^{{2,\alpha}}(\partial \Omega_{+})}+\norm{f^{-}}_{C^{{2,\alpha}}(\partial \Omega_{-})}+ \abs{J-J_0} \leq KM ,
\end{equation} 
there exists a unique $(B,p)\in C^{{2,\alpha}}(\Omega)\times C^{{2,\alpha}}(\Omega)$ to \eqref{MHS2D:current}  with $\norm{B-B_0}_{C^{2,\alpha}(\Omega)}\leq M$ such that
\begin{align*}
p+\frac{\abs{B}^2}{2}&=p_0+\frac{\abs{B_0}^2}{2}+h^{-}  \ \mbox{on} \ \partial \Omega_{-}, \partial_{x}(p+\frac{\abs{B}^2}{2})=\partial_{x}(p_0+\frac{\abs{B_0}^2}{2})+\partial_{x}h^{+}  \ \mbox{on} \ \partial \Omega_{+}\\
 B\cdot n &=B_{0}\cdot n+f^{-}\ \mbox{on} \ \partial \Omega_{-}, \ \mbox{ and } \int_{\mathcal{C}}B\cdot n \ dS= J.
\end{align*}
The constants $M,K,$ as well as $\epsilon,$ depend  only on $\alpha,L, \bar{B}^{2}_{0}$.
\end{theorem}

\begin{theorem}[Case \textit{C}. Compatibility condition]\label{Th:case2:C2:MHS}
 Let $\Omega=\{ (x,y)\in \mathbb{S}^{1}\times (0,L)\}$, with $L>0$ and $\alpha\in (0,1)$.  Suppose that $(B_0,p_0)\in C^{{2,\alpha}}(\Omega)\times C^{{2,\alpha}}(\Omega)$ is a solution of \eqref{MHS2D:current} with $\bar{B}^{2}_{0}=\displaystyle\inf_{(x,y)\in\bar{\Omega}} \abs{B^2_{0}(x,y)}> 0$ and $\mbox{curl }B_0=0$. For $\mathcal{C}= \{(0,y), y\in[0,L] \}$ we have that the integral $\int_{\mathcal{C}} (B_{0}\cdot n) \ dS$ is a real  constant that we will denote as $J_0$. There exist $\epsilon>0$ , $M>0$ sufficiently small as well as $K>0$ such that for 
$B_0$ as above with $ \norm{B^1_0}_{C^{2,\alpha}(\Omega)}\leq\epsilon $ and  $h^{-}\in C^{{2,\alpha}}(\partial \Omega_{-})$,  $h^{+}\in C^{{2,\alpha}}(\partial \Omega_{+}), f^{-}\in C^{{2,\alpha}}(\partial \Omega_{-})$ and $J\in \mathbb{R}$  satisfying 
\begin{equation}
\norm{h^-}_{C^{{2,\alpha}}(\partial \Omega_{-})} + \norm{h^+}_{C^{{2,\alpha}}(\partial \Omega_{+})}+\norm{f^{-}}_{C^{{2,\alpha}}(\partial \Omega_{-})}+ \abs{J-J_0} \leq KM ,
\end{equation} 
there exists a unique $(B,p)\in C^{{2,\alpha}}(\Omega)\times C^{{2,\alpha}}(\Omega)$ to \eqref{MHS2D:current}  with $\norm{B-B_0}_{C^{2,\alpha}(\Omega)}\leq M$ such that
\begin{equation}
p =p_{0}+h^- \ \mbox{on} \ \partial \Omega_{-},  p =p_{0}+h^{+} \mbox{on} \ \partial \Omega_{+}, \ B\cdot n=B_{0}\cdot n+f ^{-} \ \mbox{on} \ \partial \Omega_{-} \mbox{ and } \int_{\mathcal{C}}B\cdot n \ dS= J
\end{equation}
if and only if 
\begin{equation}
h{+}(0)= \Lambda(h^{-},h^{+},f^{-},J).
\end{equation}
\end{theorem}

\begin{theorem}[Case \textit{D}]\label{theorem:D:MHS}
Let $f\in C^{1,\alpha}(\partial\Omega_{-}), h^{-}\in C^{1,\alpha}(\partial\Omega_{-})$, $h^{+}\in C^{1,\alpha}(\partial\Omega_{+})$ and $h^{+}=h^{-}\circ T$ where $T:\mathbb{S}^1\to \mathbb{S}^1$ is a given diffeomorphism with $C^{2,\alpha}$ regularity.  Then if $f>0$ for $x\in \partial \Omega_{-}$, there exists a solution $(B,p)\in C^{{1,\alpha}}(\Omega)\times C^{{1,\alpha}}(\Omega)$ solving the MHS equation \eqref{MHS2D:current} such that 
\begin{equation}
B\cdot n =f \mbox{on} \ \partial \Omega^{-}, \ p=h^{-}  \ \mbox{on} \ \partial \Omega_{-} \mbox{ and } p=h^{+}  \ \mbox{on} \ \partial \Omega_{+}.
\end{equation}
Moreover, there exists $\delta>0$ such that if 
\begin{equation}
\norm{h^{-}}_{C^{{1,\alpha}}(\partial \Omega_{-})}+\norm{f}_{C^{{1,\alpha}}(\partial \Omega)}\leq  \delta ,
\end{equation}
the solution $(B,p)$ is unique. 
\end{theorem}

\begin{theorem}[Case \textit{G}]\label{theorem:G:MHS}
 Let $\Omega=\{ (x,y)\in \mathbb{S}^{1}\times (0,L)\}$, with $L>0$ and $\alpha\in (0,1)$.  Suppose that $(B_0,p_0)\in C^{{2,\alpha}}(\Omega)\times C^{{2,\alpha}}(\Omega)$ is a solution of \eqref{MHS2D:current} with $\bar{B}^{2}_{0}=\displaystyle\inf_{(x,y)\in\bar{\Omega}} \abs{B^2_{0}(x,y)}> 0$ and $\mbox{curl }B_0=0$. For $\mathcal{C}= \{(0,y), y\in[0,L] \}$ we have that the integral $\int_{\mathcal{C}} (B_{0}\cdot n) \ dS$ is a real  constant that we will denote as $J_0$. There exist $\epsilon>0$ , $M>0$ sufficiently small as well as $K>0$ such that for 
$B_0$ as above with $ \norm{B^1_0}_{C^{2,\alpha}(\Omega)}\leq\epsilon $ and $h^{-}\in C^{{2,\alpha}}(\partial \Omega_{-})$,  $h^{+}\in C^{{2,\alpha}}(\partial \Omega_{+}), f^{-}\in C^{{2,\alpha}}(\partial \Omega_{-})$ and $J\in \mathbb{R}$  satisfying 
\begin{equation}
\norm{h^-}_{C^{{2,\alpha}}(\partial \Omega_{-})} + \norm{h^+}_{C^{{2,\alpha}}(\partial \Omega_{+})}+\norm{f^{-}}_{C^{{2,\alpha}}(\partial \Omega_{-})}+ \abs{J-J_0} \leq KM ,
\end{equation} 
there exists a unique $(B,p)\in C^{{2,\alpha}}(\Omega)\times C^{{2,\alpha}}(\Omega)$ to \eqref{MHS2D:current}  with $\norm{B-B_0}_{C^{2,\alpha}(\Omega)}\leq M$ such that
\begin{equation}
p =p_{0}+h^- \ \mbox{on} \ \partial \Omega_{-},  p+\frac{\abs{B}^2}{2}=p_{0}+\frac{\abs{B_0}^2}{2}+h^{+} \mbox{on} \ \partial \Omega_{+}, \ B\cdot n=B_{0}\cdot n+f ^{-} \ \mbox{on} \ \partial \Omega_{-} \mbox{ and } \int_{\mathcal{C}}B\cdot n \ dS= J.
\end{equation}
The constants $M,K,$ as well as $\epsilon,$ depend  only on $\alpha,L, \bar{B}^{2}_{0}$.
\end{theorem}


\vspace{.2in}
\noindent{\bf{Acknowledgment.}} 
D. Alonso-Or\'an is supported by the Alexander von Humboldt Foundation. J. J. L.  Vel\'azquez acknowledges support through the CRC 1060 (The Mathematics of Emergent Effects) that is funded through the German Science Foundation (DFG), and the Deutsche Forschungsgemeinschaft (DFG, German Research Foundation) under Germany ’s Excellence Strategy – EXC-2047/1 – 390685813.

\end{document}